\documentclass[leqno]{amsart}

\usepackage{amssymb,amsfonts,amsthm}%,psfig}
\usepackage{graphicx}
\usepackage{subfigure}
\usepackage{amsthm}

\usepackage{amsmath}
\usepackage{bm}             %% boldmath-command: \bm{}
\usepackage[mathscr]{eucal}
\usepackage{color}
\usepackage{cancel}
\usepackage{enumerate}
\usepackage{psfrag}
\usepackage{bbm}
\usepackage{appendix}

\newcommand{\R}{\mathbb{R}}
\newcommand{\bfR}{\mathbb{R}}

\newcommand{\Omv}{\Omega_v^\alpha}
\newcommand{\Omy}{\Omega_y^\alpha}
\newcommand{\Omz}{\Omega_z^\alpha}
\newcommand{\Om}{\Omega}

\newcommand{\chfn}{\mathbbm{1}}
\newcommand{\mcA}{\mathcal{A}}
\newcommand{\mcM}{\mathcal{M}}
\newcommand{\mcS}{\mathcal{S}}
\newcommand{\mcD}{\mathcal{D}}
\newcommand{\mcG}{\mathcal{G}}

\newcommand{\mcJ}{\mathcal{J}}
\newcommand{\mcP}{\mathcal{P}}
\newcommand{\mcW}{\mathcal{W}}
\newcommand{\mcX}{\mathcal{X}^\alpha}
\newcommand{\mcV}{\mathcal{V}^\alpha}
\newcommand{\mcY}{\mathcal{Y}^\alpha}
\newcommand{\mcZ}{\mathcal{Z}^\alpha}
\newcommand{\mcU}{\mathcal{U}}

\newcommand{\bS}{\mcS^\alpha}

\newcommand{\mcEVP}{\mathcal{E}_\mathrm{VP}}

\newcommand{\alt}[1]{#1}

\newtheorem{theorem}{Theorem}[section]
\newtheorem*{theorem*}{Theorem}
\newtheorem{corollary}[theorem]{Corollary}

\newtheorem{lemma}[theorem]{Lemma}

\theoremstyle{definition}

\newtheorem{remark}{Remark}[section]

\usepackage{fullpage}

\title{Asymptotic Dynamics of Dispersive, Collisionless Plasmas}
\author{Stephen Pankavich}
\address{Department of Applied Mathematics and Statistics, Colorado School of Mines, Golden, CO 80401.}
\email{pankavic@mines.edu}

\date{\today}
\thanks{The author was supported in part by NSF grants DMS-1911145 and DMS-2107938.}

\begin{document}
\maketitle

\begin{abstract}
A multispecies, collisionless plasma is modeled by the Vlasov-Poisson system. 
Assuming that the electric field decays with sufficient rapidity as $t \to\infty$, we show that the velocity characteristics and spatial averages of the particle distributions converge as time grows large.
Using these limits we establish the precise asymptotic profile of the electric field and its derivatives, as well as, the charge and current densities.
Modified spatial characteristics are then shown to converge using the limiting electric field.
Finally, we establish a modified $L^\infty$ scattering result for each particle distribution function, namely we show that they converge as $t \to \infty$ along the modified spatial characteristics.
When the plasma is non-neutral, the estimates of these quantities are sharp, while in the neutral case they may imply faster rates of decay.

\end{abstract}

\section{Introduction}
We consider a plasma comprised of a large number of charged particles. If there are $N$ distinct species of charge within the plasma, particles of the $\alpha$th species (for $\alpha = 1, ..., N$)
have charge \alt{$q_\alpha \in \mathbb{R}$, mass $m_\alpha > 0$,} and are distributed in phase space at time $t \geq 0$ according to the function $f^\alpha(t,x,v)$ where $x \in \bfR^3$ represents particle position and $v\in \bfR^3$ particle velocity. 
Assuming that electrostatic forces dominate collisional effects, the time evolution of the plasma is described by the multispecies Vlasov-Poisson system
\begin{equation}
\tag{VP}
\label{VP}
\left \{ \begin{aligned}
& \partial_{t}f^\alpha+v\cdot\nabla_{x}f^\alpha+ \alt{\frac{q_\alpha}{m_\alpha}} E \cdot\nabla_{v}f^\alpha=0, \qquad \alpha = 1, ..., N\\
& \rho(t,x)= \sum_{\alpha=1}^N \alt{q_\alpha} \int_{\mathbb{R}^3} f^\alpha(t,x,v) \,dv\\
& E(t,x) = \nabla_x ( \Delta_x)^{-1} \rho(t,x) = \alt{\frac{1}{4\pi}}\int_{\mathbb{R}^3} \frac{x-y}{\vert x - y \vert^3} \rho(t,y) \ dy
\end{aligned} \right .
\end{equation}
with prescribed initial conditions $f^\alpha(0,x,v) = f^\alpha_0(x,v)$.
%Additionally, the electrostatic potential can be defined by
%$$\phi(t,x) = \Delta^{-1}_x \rho(t,x)$$
%so that the induced field satisfies
%$$E(t,x) = \nabla_x \phi(t,x).$$
Here, $E(t,x)$ represents the electric field induced by the charged particles, and $\rho(t,x)$ is the associated density of charge within the plasma.
Additionally, the corresponding current density is defined by
$$j(t,x) = \sum_{\alpha=1}^N \alt{q_\alpha} \int v f^\alpha(t,x,v)  \ dv.$$
%For simplicity, all particle masses have been normalized above. %, as the results we will present later are unaffected by particle masses.
%
\alt{For each species, the particle number and velocity} are conserved, namely
$$ \iint f^\alpha(t,x,v) \ dvdx =  \iint f^\alpha_0(x,v) \ dvdx =: \mcM^\alpha$$
and
$$ \iint v f^\alpha(t,x,v) \ dvdx =  \iint v f^\alpha_0(x,v) \ dvdx =: \mcJ^\alpha,$$
for all $t \geq 0$, with the overall net charge and total momentum given by
$$\mcM = \sum_{\alpha = 1}^N \alt{q_\alpha} \mcM^\alpha \qquad \mathrm{and} \qquad \mcJ = \sum_{\alpha = 1}^N \alt{m_\alpha} \mcJ^\alpha.$$
The total energy of the system is also conserved, i.e. for all $t \geq 0$
\begin{align*}
& \sum_{\alpha=1}^N \iint \frac{1}{2}\alt{m_\alpha} |v|^2 f^\alpha(t,x,v) \ dx \ dv + \frac{1}{2} \int |E(t,x)|^2 \ dx \\
& = \sum_{\alpha=1}^N \iint \frac{1}{2} \alt{m_\alpha} |v|^2 f^\alpha_0(x,v) \ dx \ dv + \frac{1}{2} \int |E(0,x)|^2 \ dx =: \mcEVP.
\end{align*}

It is well-known that given smooth initial data \alt{with compact support in phase space or finite moments}, \eqref{VP} possesses a smooth global-in-time solution \cite{LP, Pfaf, Schaeffer}. Such global existence results often depend upon either the propagation of (spatial, velocity, or transported) moments or precise estimates for the growth of the characteristics associated to \eqref{VP}, which are defined by
\begin{equation}
\label{char}
\left \{
\begin{aligned}
&\dot{\mcX}(t, \tau, x, v)=\mcV(t, \tau, x, v)\\
&\dot{\mcV}(t, \tau, x, v)= \alt{\frac{q_\alpha}{m_\alpha}} E(t, \mcX(t, \tau, x, v))
\end{aligned}
\right.
\end{equation}
with initial conditions
$\mcX(\tau, \tau, x, v) = x$ and
$\mcV(\tau, \tau, x, v) = v.$
%When describing the particle characteristics in the future, we will typically shorten the notation so that
%$$\mcX(s) = \mcX(s,t,x,v)  \qquad \mathrm{and} \qquad \mcV(s) = \mcV(s,t,x,v)$$
%will denote these functions.
%
For additional background, we refer the reader to \cite{Glassey, Rein} as general references concerning \eqref{VP} and associated kinetic equations. 

Though well-posedness has been intensely studied, the large time behavior of solutions of \eqref{VP} is less understood. Partial results for the Cauchy problem are known in some special cases, including small data \cite{BD, HRV, Ionescu, Smulevici}, monocharged and spherically-symmetric data \cite{Horst, Pankavich2020}, and lower-dimensional settings \cite{BKR, GPS, GPS2, GPS4, Sch}.
More recently, an understanding of the intermediate asymptotic behavior was obtained in \cite{BCP1}\alt{, namely that there are solutions for which the $L^\infty$ norms of the charge density and electric field can be made arbitrarily large at some later time regardless of their initial size}.
\alt{Due to the dispersive properties imparted upon the system by the transport operator $\partial_t + v \cdot \nabla_x$ and the velocity averaging inherent to the electric field and charge density, one expects that these quantities decay to zero like $t^{-2}$ and $t^{-3}$, respectively, as $t \to \infty$ for all smooth solutions of \eqref{VP}}. In fact, both small data and spherically-symmetric solutions have been shown to exhibit exactly this behavior, and it is known that the Cauchy problem does not possess smooth steady states (cf., \cite{GPS5}).

As we will show, the dynamical behavior of the system depends crucially upon decay of the electric field. 
To date, the best \emph{a priori} rate of decay known \cite{Yang} for the electric field in \eqref{VP} is
$$\| E(t) \|_\infty \leq C(1+t)^{-1/6},$$
and this is derived from precise estimates of the growth of the maximal velocity on the support of $f(t)$, from which the estimate of $\|E(t)\|_\infty$ is deduced via Young's inequality. 
Unfortunately, this resulting estimate is far from optimal.
Additionally, while maximal velocity support estimates have been beneficial to improving the field decay rate \cite{Chen2, Pallardsupport, Jacksupport}, even a sharp estimate (i.e, a uniform bound) on the support cannot allow one to conclude a sharp decay rate of the field.
So, it appears that we are quite far from obtaining precise estimates of the field. %, which is crucial in understanding the behavior of the remaining quantities in the system.
Still, even when the field decays at a fast rate, the dynamics of the remaining quantities in \eqref{VP} are not well-understood.
Hence, the goal of the current work is to establish the precise large-time dynamical behavior of solutions to \eqref{VP} whenever the electric field is known to decay with sufficient rapidity.

\subsection{Overview and Organization}

Due to the global existence theorem, all quantities of interest are bounded for finite time; thus, we are only concerned with large time estimates.
Hence, we use the notation
$$A(t) \lesssim B(t)$$
to represent the statement that there is $C > 0$
such that
$A(t) \leq CB(t)$
for $t$ sufficiently large.
When necessary, $C$ will denote a positive constant (independent of the solution) that may change from line to line.
%\section{Large-time Behavior and Particle Dynamics}

%Having obtained the optimal decay estimate of the field in $L^\infty$ within the previous section, we may now focus on other quantities in the system. First, we investigate the behavior of characteristics and the maximal particle position and velocity on the support of $f$, and then use these to prove sharp lower bounds on the charge density and electric field.
%
Throughout we let $\mcU \subset \bfR^6$ represent an arbitrary compact set, take $f^\alpha_0 \in C_c^2(\bfR^6)$, and let $f^\alpha(t,x,v)$ denote the corresponding $C^2$ solution of \eqref{VP}.
We assume that the dispersion in the system induces strong decay of the electric field, namely there is $p \in \left (\frac{5}{3}, 2 \right]$ such that
%$C > 0$ and
\begin{equation}
\label{Assumption}
\tag{A}
\|E(t) \|_\infty \lesssim t^{-p}.
\end{equation}
%for $t \geq 0$. 
We note that this assumption is known to be satisfied for monocharged, spherically-symmetric initial data \cite{Horst, Pankavich2020} and for all previously constructed perturbative solutions (e.g., \cite{BD}).

%We note that the optimal rates in \eqref{Edecay} and \eqref{rhodecay} can be relaxed, as decay rates of $\|E(t) \|_\infty \lesssim t^{-p}$ or $\|\rho(t) \|_\infty \lesssim t^{-q}$ for some $p > 5/3$ or $q > 2$ will imply the above decay rates.
%Though we will not prove the result in \cite{Jacknew}, we include a similar, but less complex, result (Lemma \ref{Erho}) that ensures nearly optimal decay.

\alt{Though we will assume compactly-supported initial data, this may not be necessary as velocity, spatial, and transported moments \cite{Castella, Chen2, LP, Pallardspatial, Pallardsupport} have all be used in lieu of this assumption to develop the existence theory.
In addition, the regularity assumptions on initial data may be weakened to arrive at similar convergence results in weaker topologies (see \cite{Ionescu}).}
The novelty herein is that a modified \alt{scattering} result for classical solutions of a multispecies plasma is established and the precise asymptotic profile of the electric field, its derivatives, and the charge and current density are obtained.  
Additionally, \alt{as in \cite{Ionescu}, we further} demonstrate the importance of the velocity-dependent density $\mcP(t,v)$, defined in Theorem \ref{T1}, rather than the charge density, to understanding the large time behavior of the system, and require only $C^2$ initial data rather than many higher derivatives in $L^2$. 
Finally, our results apply directly to general conditions that may be satisfied by any plasma and not merely to \alt{small data} solutions. \alt{Furthermore, our methods do not depend upon the sign of the force field, only its strength for large times, and therefore the tools established herein also apply to solutions of the gravitational Vlasov-Poisson system that satisfy \eqref{Assumption}.}
%For instance, our results apply to solutions launched from monocharged, spherically-symmetric initial data, though the distribution function, density, and field limits will be supported on a lower-dimensional subset of phase space.

\subsection{Main Results}
For $t \geq 0$ define the support of $f^\alpha(t)$ by
$$\alt{\bS_f}(t) = \overline{\left \{ (x,v) \in \R^6 : f^\alpha(t,x,v) \neq 0 \right \}}.$$
Our main results can now be stated precisely.
\begin{theorem}
\label{T1}
Assume \eqref{Assumption} holds. Then, we have the following: 
\begin{enumerate}%[(A)]
\item 
For every $\alpha = 1, ..., N$, $\tau \geq 0$ and $(x,v) \in \mcU$ the limiting function $\mcV_\infty$ defined by
$$\mcV_\infty(\tau, x, v) :=  \lim_{t \to \infty} \mcV(t, \tau, x, v)$$ 
exists and is $C^2$ and bounded.
Additionally, for $\tau \geq 0$ and $(x,v) \in \mcU$,
$$\left | \mcV(t, \tau, x, v) - \mcV_\infty(\tau, x, v)  \right | \lesssim t^{-1}.$$

\item 
For every $\alpha = 1, ..., N$ define
$$\Omv = \left \{ \mcV_\infty(0, x, v) : (x, v) \in \alt{\bS_f(0)} \right \}.$$
Then, there exist $F^\alpha_\infty \in C_c^2(\bfR^3)$ supported on $\Om_v^\alpha$ such that the spatial average
%with $\mathrm{supp}(F_\infty) = \Omv$ depending only upon the limiting particle velocities
$$F^\alpha(t,v) = \int f^\alpha(t,x, v) \ dx$$
satisfies $F^\alpha(t,v) \to F^\alpha_\infty(v)$ uniformly as $t \to \infty$, namely
$$\| F^\alpha(t) - F^\alpha_\infty \|_\infty \lesssim t^{-1}\ln^{\alt{4}}(t).$$
%
%In particular, $F_\infty$ conserves mass, momentum, and energy, i.e.
%$$\int F^\alpha_\infty(v) \ dv = \mcM^\alpha, \quad
%\int v F^\alpha_\infty(v) \ dv = \mcJ^\alpha, \quad \mathrm{and} \qquad
%\sum_{\alpha=1}^N \int \frac{1}{2} |v|^2 F^\alpha_\infty(v) \ dv = \mcEVP.$$
Moreover, the net density
$$\mcP(t,v) = \sum_{\alpha=1}^N  \alt{q_\alpha} F^\alpha(t,v)$$
converges uniformly to the limit
$$\mcP_\infty(v) = \sum_{\alpha=1}^N  \alt{q_\alpha} F^\alpha_\infty(v)$$
at the same rate, and satisfies
\begin{equation}
\label{Pinfmass}
\int \mcP_\infty(v) \ dv = \mcM.
\end{equation}

\item Define $E_\infty(v) = \nabla_v(\Delta_v)^{-1} \mcP_\infty(v)$. Then, we have the self-similar asymptotic profiles
\begin{align*}
\sup_{x\in \bfR^3} \left | t^2 E(t,x) - E_\infty \left (\frac{x}{t} \right ) \right | & \lesssim t^{-1}\ln^{\alt{4}}(t),\\
\sup_{x \in \bfR^3}  \left | t^3 \nabla_xE(t,x) - \nabla_v E_\infty \left (\frac{x}{t} \right) \right | & \lesssim t^{-1} \ln^{\alt{6}}(t),\\
\sup_{x \in \bfR^3}   \left | t^3 \rho(t,x) - \mcP_\infty \left (\frac{x}{t} \right) \right | & \lesssim t^{-1} \ln^{\alt{5}}(t),\\
\sup_{x \in \bfR^3}  \left | t^3 j(t,x) - \frac{x}{t}\mcP_\infty\left (\frac{x}{t} \right) \right | & \lesssim t^{-1}\ln^{\alt{5}}(t).
\end{align*}

\item 
For any $\alpha = 1, ..., N$, $\tau, t \geq 1$ and $(x,v) \in \mcU$, define
$$\mcZ(t,\tau, x, v) = \mcX(t, \tau, x, v) - t \mcV(t, \tau, x, v) + \alt{q_\alpha} \ln(t) E_\infty(\mcV(t,\tau, x, v) ).$$ 
Then, the limiting function $\mcZ_\infty$ given by
$$\mcZ_\infty(\tau, x, v) :=  \lim_{t \to \infty} \mcZ(t, \tau, x, v)$$ 
exists and is $C^2$ and bounded.
Additionally, for any $\tau \geq 0$ and $(x,v) \in \mcU$,\
$$|\mcZ(t, \tau, x, v) - \mcZ_\infty(\tau, x, v) | \lesssim t^{-1}\ln^{\alt{4}}(t).$$
%and
%$$ \left |\mcZ_\infty(\tau, x, v) - \left (x-v\tau + \alt{q_\alpha} \ln(\tau) E_\infty(v) \right ) \right | \lesssim \tau^{-1}\ln^2(\tau).$$

\item
For every $\alpha = 1, ..., N$ define
$$\Omz = \left \{ \mcZ_\infty(1, x, v) : (x, v) \in \alt{\bS_f(1)} \right \}$$
and
$\Om^\alpha = \Omz \times \Omv$.
Then, there is $f^\alpha_\infty \in C_c^2(\bfR^6)$ supported on $\Om^\alpha$ such that
$$f^\alpha \left (t,x +vt - \alt{\frac{q_\alpha}{m_\alpha}} \ln(t)E_\infty(v),v \right) \to f^\alpha_\infty(x,v)$$
uniformly
as $t \to \infty$, namely we have the convergence estimate
$$\sup_{(x,v) \in \bfR^6} \left | f^\alpha \left (t,x +vt - \alt{\frac{q_\alpha}{m_\alpha}} \ln(t)E_\infty(v), v \right) - f^\alpha_\infty(x,v) \right |  \lesssim t^{-1}\ln^{\alt{4}}(t).$$
In particular, $f^\alpha_\infty$ conserves \alt{particle number, velocity}, and energy, i.e.
$$\iint f^\alpha_\infty(x,v) \ dv dx = \mcM^\alpha, \ 
\iint v f^\alpha_\infty(x,v) \ dv dx = \mcJ^\alpha, \ \mathrm{and} \ 
\sum_{\alpha=1}^N \iint \frac{1}{2} \alt{m_\alpha}|v|^2 f^\alpha_\infty(x,v) \ dv dx = \mcEVP.$$
\end{enumerate}
\end{theorem}

If the plasma is non-neutral, i.e. $\mcM \neq 0$, then $\mcP_\infty \not\equiv 0$ due to \eqref{Pinfmass} and these estimates are sharp, up to a correction in the logarithmic powers of the error terms (see Section \ref{sect:prop}).
However, when the plasma is neutral, i.e. $\mcM = 0$, it is possible that the limiting density $\mcP_\infty$ (and hence the limiting field $E_\infty$) is identically zero, which implies stronger decay of these quantities.
\alt{This could be} related to the phenomenon of Landau Damping \cite{VM}, which \alt{can give rise to decay that is faster than the dispersive rates and} has recently been shown to occur for unconfined, but screened, neutral plasmas that are sufficiently close to spatially-homogeneous equilibria \cite{BMM, HNR}.

\begin{theorem}
\label{T2}
If $\mcM = 0$ and $\mcP_\infty \equiv 0$, then the asymptotic behavior described above is altered in the following manner:
\begin{enumerate}
\item 
For any $\alpha = 1, ..., N$, $\tau \geq 0$ and $(x,v) \in \mcU$, we have
$$\left | \mcV(t, \tau, x, v) - \mcV_\infty(\tau, x, v)  \right | \lesssim t^{-2}.$$

\item We have the faster decay estimates
\begin{align*}
\| F^\alpha(t) - F^\alpha_\infty \|_\infty & \lesssim t^{-2}, & 
\| \mcP(t) \|_\infty & \lesssim t^{-2}, \\
\| E(t) \|_\infty & \lesssim t^{-3}, & 
\| \nabla_xE(t) \|_\infty & \lesssim t^{-4} \ln(t),\\
\| \rho(t) \|_\infty & \lesssim t^{-4}, & 
\| j(t) \|_\infty & \lesssim t^{-4}.
\end{align*}

\item For any $\alpha = 1, ..., N$, the spatial characteristics
$$\mcY(t,\tau, x, v) = \mcX(t, \tau, x, v) - t \mcV(t, \tau, x, v),$$ 
converge as $t \to \infty$, and the limiting functions 
$$\mcY_\infty(\tau, x, v) :=  \lim_{t \to \infty} \mcY(t, \tau, x, v)$$
are $C^2$ and bounded.
Additionally, for any $\tau \geq 0$ and $(x,v) \in \mcU$,\
$$|\mcY(t, \tau, x, v) - \mcY_\infty(\tau, x, v) | \lesssim t^{-1}.$$

\item
For any $\alpha = 1, ..., N$ define
$$\Omy = \left \{ \mcY_\infty(0, x, v) : (x, v) \in \alt{\bS_f(0)} \right \}$$
and
$\Om^\alpha = \Omy \times \Omv$.
Then, there is $f^\alpha_\infty \in C_c^2(\bfR^6)$ supported on $\Om^\alpha$ such that
$$f^\alpha(t,x +vt,v) \to f^\alpha_\infty(x,v)$$
uniformly
as $t \to \infty$, namely we have the convergence estimate
$$\sup_{(x,v) \in \bfR^6} \left | f^\alpha(t,x +vt, v) - f^\alpha_\infty(x,v) \right |  \lesssim t^{-1}$$
\alt{and again, the particle number, velocity, and energy are all conserved in the limit.}
\end{enumerate}
\end{theorem}
We note that solutions with a trivial electric field and charge density can be constructed when $\mcM = 0$, but not when $\mcM \neq 0$, and this accounts for the contrasting asymptotic behavior of the systems.
In particular, considering the simplified regime in which only electrons and a single species of ions \alt{with opposite charge} are considered, namely $N=2$, $\alt{q_1} = +1$, and $\alt{q_2} = -1$, one may take $f^1_0(x,v) = f^2_0(x,v) \in C^2_c(\bfR^6)$.
This implies 
$$ \rho(0, x) = \int \left ( f^1_0(x,v) - f^2_0(x,v) \right ) \ dv = 0$$
for all $x \in \bfR^3$, and thus
$\mcM = 0$ and $E(0,x) \equiv 0$. As the two distribution functions then satisfy the same evolution equation (with zero electric field) and the same initial data, uniqueness guarantees $\rho(t,x) = E(t,x) = 0$ for all $t \geq 0, x \in \bfR^3$.
Hence, $\mcP_\infty \equiv E_\infty \equiv 0$, and this example demonstrates that the decay estimates provided by Theorem \ref{T2} need not be optimal in general.
Such charge cancellation may also occur in the limit rather than choosing the charge density and electric field to be identically zero for all time.

Contrastingly, a neutral plasma may display the same asymptotic behavior as in the case $\mcM \neq 0$, whenever $\mcP_\infty \not\equiv 0$.
% the asymptotic behavior of the system may resemble that of Theorem \ref{T1} rather than Theorem \ref{T2}.  
In particular, considering again $N=2$, $\alt{q_1} = +1$, and $\alt{q_2} = -1$, the initial distribution of positive ions and electrons can be widely separated in space with sufficiently large velocities in opposing directions. Then, the field will be sufficiently weak so that the plasma will behave like two non-interacting monocharged species, and the asymptotic behavior will be analogous to the non-neutral case. Here, the two spatial averages $F^1(t,v)$ and $F^2(t,v)$ can be structured to converge to different limiting functions $F^1_\infty(v) \neq F^2_\infty(v)$, and thus
$$\mcP_\infty(v) = F^1_\infty(v) - F^2_\infty(v) \not \equiv 0.$$

\begin{remark}
Rather than imposing \eqref{Assumption}, one may instead require a decay condition on the charge density, namely
$\|\rho(t) \|_\infty \lesssim t^{-q}$ for some $q \in (2,3]$, or a growth condition on the spatial support of the translated distribution function defined in the next section, 
i.e. $\alt{\mu(t)} \lesssim t^{r}$ for some $r \in \left ( 0, \frac{1}{2} \right )$.
\end{remark}

\begin{remark}
We note that our results can be extended to higher ($d \geq \alt{4}$) dimensions, as well, with the same decay assumption and methods. %\alt{\cite{Pankavich2021}}.
\end{remark}

%\begin{exmp} [Minimal Field Interaction]
%Let $\epsilon = 10^{-6}$ and $\alt{q_1} = (1,0,0)$ with $\chi_\epsilon(x; a) \in C_c^2(\bfR^3)$ satisfying $\mathrm{supp}(\chi) \subset B_\epsilon(a)$, $0 \leq \chi \leq 1$, and $\int \chi(x; a) \ dx = 1$ for a given $a \in \bfR^3$.
%
%Then, define $f^1_0(x,v) = \epsilon \chi \left (x; \frac{\alt{q_1}}{\epsilon^3} \right) \chi \left (v; \frac{\alt{q_1}}{\epsilon^3} \right)$ and $f^2_0(x,v) =  \epsilon\chi \left (x; -\frac{\alt{q_1}}{\epsilon^3} \right)  \chi \left (v; -\frac{\alt{q_1}}{\epsilon^3} \right)$
%so that the initial ion and electron distribution functions are small, widely separated in space, and moving with large velocity in opposing directions. Due to their distance, they will feel only a very small attraction from their associated electric field. Because of their large velocities in opposite directions, the field will be forced to zero in small time.
%\end{exmp}

\subsection{Strategy of the Proof}
To establish the theorems we will use both Lagrangian and Eulerian methods, as the latter allow us to prove sharp estimates and establish the existence of limits, while the former captures the properties of limiting functions.
In the Eulerian framework, we reformulate the original problem within a dispersive reference frame that is co-moving with the particles.
More specifically, let 
$$g^\alpha(t,x,v) = f^\alpha(t, x+vt, v)$$
for $\alpha = 1, ..., N$ and apply a change of variables (see the proof of Lemmas \ref{LField} and \ref{LDensity}) to the field and charge density so that \eqref{VP} becomes
\begin{equation}
\tag{VP$_g$}
\label{VPg}
\left \{ \begin{aligned}
& \partial_{t}g^\alpha - \alt{\frac{q_\alpha}{m_\alpha}} t E(t,x+vt)\cdot \nabla_{x}g^\alpha+ \alt{\frac{q_\alpha}{m_\alpha}} E(t,x+vt) \cdot\nabla_{v}g^\alpha=0, \qquad \alpha = 1, ..., N\\
& E(t,x+vt) = \alt{\frac{1}{4\pi t^2}} \sum_{\alpha=1}^N  \alt{q_\alpha} \iint \frac{\xi}{|\xi|^3} \ g^\alpha \left (t, w, v - \xi + \frac{x-w}{t} \right ) \ dwd\xi
\end{aligned} \right .
\end{equation}
with
$$\rho(t,x+vt)= t^{-3} \sum_{\alpha=1}^N \alt{q_\alpha} \int_{\mathbb{R}^3} g^\alpha \left (t,w,v + \frac{x-w}{t} \right)\,dw$$
and the initial conditions $g^\alpha(0,x,v) = f^\alpha_0(x,v)$.
Then, as we will show, $g^\alpha$ possesses nicer properties than the original distribution function $f^\alpha$. Indeed, the spatial support of $g^\alpha$ grows significantly slower than that of $f^\alpha$ and, due to the cancellation of derivative terms, the $v$-derivatives of $g^\alpha$ can grow logarithmically in time (Lemma \ref{Dvg}), while the corresponding $v$-derivatives of $f^\alpha$ grow linearly at best (Lemma \ref{DEf}).
In particular, we note that the convolution in the electric field is now in the velocity variable rather than the spatial variable. 
%In this direction, we define the net spatial density
%$$\mcP(t,v) = \sum_{\alpha = 1}^N  \alt{q_\alpha} \int_{\mathbb{R}^3} g^\alpha \left (t,w,v \right)\,dw$$ 
%and note that 
Hence, as $t \to \infty$ one expects 
$$E(t,x+vt) \sim \alt{\frac{1}{4\pi t^2}} \sum_{\alpha=1}^N \alt{q_\alpha} \iint \frac{\xi}{|\xi|^3} g^\alpha(t, w, v-\xi) \ dw d\xi = t^{-2}\nabla_v \left (\Delta_v \right )^{-1} \mcP (t,v)$$
locally in $x$.
With this, it becomes clear that estimates of the spatial support of $g^\alpha$ and its velocity derivatives will be instrumental, while estimates of the velocity support %, which have garnered immense attention over the years (e.g., \cite{Pallardsupport, Jacksupport, Yang}) 
may be less important.
Thus, our results may provide better tools to obtain \emph{a priori} estimates on the spatial growth (via either the support or moments) of $g^\alpha$ and velocity derivatives $\nabla_v g^\alpha$, which are the two main ingredients in the theorems.

The general structure of the argument is as follows. From the decay of the electric field provided by \eqref{Assumption}, it is straightforward to show that every velocity characteristic has a limit as $t \to \infty$. 
These limits are then used to identify the behavior of the spatial average of each particle distribution as $t \to \infty$. 
The corresponding limiting density $\mcP_\infty$ further induces an ambient electric field, $E_\infty$, and these quantities are used to identify the precise asymptotic behavior of the field, its derivatives, and the charge and current densities.
Furthermore, the limiting electric field enables us to identify the limiting values of modified spatial characteristics.
Finally, the behavior of characteristics is used to identify the trajectories in phase space along which the particle distribution scatters to a limit.

%With suitable decay of the electric field, a higher-order correction is then introduced within the new reference frame to prove convergence of the particle distribution along a suitable approximation of the spatial characteristics.
%In particular, we define
%$$\mcZ(t,\tau, x, v) = \mcX(t, \tau, x, v) - t \mcV(t, \tau, x, v) + \alt{q_\alpha} \ln(t) E_\infty(\mcV(t,\tau, x, v) )$$ 
%and show that these characteristics converge.
%Then, the perturbed distribution function
%$$h^\alpha(t,x,v) = g^\alpha(t,x-\alt{q_\alpha} E_\infty(v) \ln(t), v) = f^\alpha(t, x+vt-\alt{q_\alpha} E_\infty(v)\ln(t), v)$$
%converges uniformly to a limiting function $f^\alpha_\infty(x,v)$ as $t \to \infty$.

In the next section, we will establish some standard estimates of derivatives of the untranslated distribution functions $f^\alpha$ and their characteristics. 
Useful properties of the translated distribution functions $g^\alpha$ are also contained in Section \ref{Lemmas}. In Section \ref{sect:vellim} we show the convergence of the velocity characteristics as $t \to \infty$, prove some properties of these limiting functions, then show the convergence of the spatial averages as $t \to \infty$ and use these to define a limiting net velocity density. Section \ref{sect:fieldconv} is dedicated to establishing the precise asymptotic behavior of the field, its derivatives, and the associated densities. The convergence of modified spatial characteristics and the modified scattering of the distribution functions rely upon these terms and are contained within Section \ref{sect:modscatt}. The proofs of Theorems \ref{T1} and \ref{T2} and a result demonstrating the sharpness of these estimates for the non-neutral case are provided in Section \ref{sect:prop}. 
%We note that the lemmas within the final section can be replaced by similar properties of the translated distribution function, but require more effort.

\section{Preliminaries}
\label{Lemmas}

\subsection{Estimates in the Original Reference Frame}
We begin by obtaining improved decay rates that follow from the main assumption \eqref{Assumption}.
First, we state a standard estimate on the gradient of the inverse Laplace operator, which will be used throughout.
\begin{lemma}[cf. \cite{Hormander}, Lemma 4.5.4]
\label{LE}
For any $\phi \in L^1(\bfR^3) \cap L^\infty(\bfR^3)$, there is $C > 0$ such that
$$\| \nabla (\Delta)^{-1} \phi \|_\infty \leq C \| \phi \|_1^{1/3}  \| \phi \|_\infty^{2/3}.$$
In particular, due to conservation of \alt{particle number}, which implies $\|\rho(t)\|_1 \leq C$ for all $t \geq 0$, we have
$$\| E(t) \|_\infty \leq C \| \rho(t) \|_\infty^{2/3}$$
for any $t \geq 0$.
\end{lemma}
This tool then allows us to address the rates of the field and density.
\alt{In particular, we will use the velocity support of $f^\alpha(t)$ to estimate the charge density.
Hence, we define 
$$\bS_v(t, x) = \overline{\left \{v \in \bfR^3 : f^\alpha(t,x,v) \neq 0 \right \} } $$
for any fixed $x \in \bfR^3$.}
\begin{lemma}
%\label{Erho}
\label{XVsupp}
Assume \eqref{Assumption}.
%there is $p > 5/3$ such that
%$$\|E(t) \|_\infty \lesssim t^{-p}.$$
Then, we have
$$\|\rho(t) \|_\infty \lesssim t^{-3} \qquad \mathrm{and} \qquad \|E(t) \|_\infty \lesssim t^{-2},$$
and
for every $\alpha = 1, ..., N$, $\tau \geq 0$, and $(x,v) \in \mcU$
$$\left | \mcV(t, \tau, x, v) \right | \lesssim 1 \quad \mathrm{and} \quad  \left | \mcX(t, \tau, x, v) \right | \lesssim t .$$
Furthermore, the velocity support of $f^\alpha$ satisfies
$$ \sup_{x \in \bfR^3} \left |\alt{\bS_v(t, x)} \right | \lesssim t^{-3}\alt{\ln^3(t)}. $$

\end{lemma}

\begin{proof}
%A variant of this lemma is essentially contained in \cite{Pankavich2020}, but we include a proof for completeness.
We begin by showing that $\alt{\bS_v(t,x)}$ ultimately lies in a ball of diameter $\mathcal{O}\left (t^{1-p} \right )$ for any $x \in \bfR^{3}$.
Suppose $(x,v_1), (x,v_2) \in \alt{\bS_f(t)}$ at some time $t \geq 1$. Integrating the characteristic equations \eqref{char} and using the bounded support of $f_0$ along with \eqref{Assumption}, we find
\begin{eqnarray*}
C & \geq & \left |\mcX(0,t, x, v_1) - \mcX(0,t,x,v_2) \right |\\
& = & \left | (v_1 - v_2) t + \alt{\frac{q_\alpha}{m_\alpha}}\int_0^t \int_s^t \biggl(E(\tau, \mcX(\tau, t, x, v_1))  - E(\tau, \mcX(\tau, t, x, v_2)) \biggr ) \ d\tau ds \right | \\
& \geq & \left | v_1 - v_2 \right | t - \alt{C}\int_0^t \int_s^t \Vert E(\tau) \Vert_\infty \ d\tau ds\\
& \geq & \left | v_1 - v_2 \right |t - C(1+t)^{2-p}
\end{eqnarray*}
for $t \geq 1$.
Rearranging the inequality produces
$$|v_1 - v_2| \leq Ct^{-1}(1 + (1+t)^{2-p}) \leq C (1+t)^{1-p}$$
for $t \geq 2$.
Hence, for any $t \geq 2$ and $x \in \bfR^3$ there are $v_0 \in \mathbb{R}^3$ and $C  > 0$ such that
$$\alt{\bS_v(t,x)} \subseteq \left \{ v \in \mathbb{R}^3 : | v - v_0| \leq C (1+t)^{1-p} \right \},$$
and thus
$$\sup_{x \in \bfR^3}\left | \alt{\bS_v(t,x)}  \right | \leq C(1+t)^{3(1-p)}$$
for any $\alpha = 1, ..., N$ and  $t \geq 2$.
Therefore, we find
$$\| \rho(t)\|_\infty  \lesssim  \sup_{x \in \mathbb{R}^3} \sum_{\alpha=1}^N \left | \int  f^\alpha(t,x,v) \ dv \right |
\lesssim \sum_{\alpha=1}^N \Vert f^\alpha_0 \Vert_\infty  \sup_{x \in \mathbb{R}^3} \left | \alt{\bS_v(t,x)} \right |
\lesssim t^{3(1-p)},$$
where $3(1-p) < -2$ as $p > 5/3$.
Next, we apply a recent result of Schaeffer to obtain the stated decay rate of the charge density: 
\begin{theorem*}[Schaeffer \cite{Jacknew}]
%\label{Schaeffer_lemma}
Assume there is $q\in(2,3]$ such that 
$$\|\rho(t)\|_{\infty}\lesssim t^{-q}.$$
Then, we have
$$\|\rho(t)\|_{\infty}\lesssim t^{-3}.$$
\end{theorem*}
With this, we also obtain 
%the assumption on the field further implies 
an improved decay estimate on the field, namely %on the field, namely
%\begin{equation}
%\label{Edecay}
$$\|E(t) \|_\infty \lesssim t^{-2},$$
%\end{equation}
which follows from Lemma \ref{LE}.
The estimate on velocity characteristics %and the velocity support 
follows directly from this as %field estimate as %\eqref{Edecay} as
$$\left |\mcV(t,\tau, x, v) \right | = \left | v + \alt{\frac{q_\alpha}{m_\alpha}}\int_\tau^t E(s, \mcX(s, \tau, x, v)) ds \right | \lesssim 1 + \int_\tau^t (1+s)^{-2} \ ds  \lesssim 1$$
for $\tau \geq 0$ and $(x,v) \in \mcU$,
and the improved decay rate for the field is merely inserted into the previous argument to deduce the claimed behavior of the velocity support. 
Finally, the estimate on the spatial characteristics of $f^\alpha$ follows by merely integrating  \eqref{char} and using the velocity bound.

\end{proof}

In addition, the decay of the field and charge density lead directly to estimates of field derivatives and derivatives of untranslated distribution functions.
\begin{lemma}
\label{DEf}
For any $\alpha = 1, ..., N$ we have %There is $C > 0$ such that
$$\|\nabla_x E(t) \|_\infty \lesssim t^{-3}\ln(t), \qquad \|\nabla_x f^\alpha(t) \|_\infty \lesssim 1, \qquad \mathrm{and} \qquad \|\nabla_v f^\alpha(t) \|_\infty \lesssim t.$$
\end{lemma}

\begin{proof}
We will utilize a standard estimate (cf. \cite[pp. 122-123]{Glassey}) on the gradient of the field, namely
\begin{equation}
\label{gradEstd}
\|\nabla_x E(t) \|_\infty \lesssim \left ( 1 + \ln^* \left (\frac{\|\nabla_x \rho(t)\|_\infty}{\|\rho(t) \|_\infty^{4/3}} \right ) \right ) \| \rho(t) \|_\infty
\end{equation}
%for $t \geq 0$
where
$$\ln^*(s) =
\begin{cases}
0, & \mathrm{if} \ s \leq 1\\
\ln(s), & \mathrm{if} \ s \geq 1.
\end{cases}
$$
We note that this bound is increasing in the contribution of $\|\rho(t)\|_\infty$ and using Lemma \ref{XVsupp} in \eqref{gradEstd} yields
$$
\|\nabla_x E(t) \|_\infty \lesssim \left ( 1 + \ln^* \left (t^4 \|\nabla_x \rho(t)\|_\infty  \right ) \right )t^{-3}.
$$
Due to Lemma \ref{XVsupp}, the velocity support of $f^\alpha$ is uniformly bounded so that
$$\left |\nabla_x \rho(t,x) \right | \lesssim \sum_{\alpha = 1}^N \int \left |\nabla_x f^\alpha(t,x,v) \right | \ dv \lesssim \max_{\alpha =1, ..., N} \|\nabla_x f^\alpha(t) \|_\infty,$$
and thus
\begin{equation}
\label{DEdecayln}
\|\nabla_x E(t) \|_\infty \leq C \left ( 1 + \ln^* \left ( \max_{\alpha =1, ..., N} \|\nabla_x f^\alpha(t)\|_\infty (1+t)^4 \right ) \right ) (1+t)^{-3}
\end{equation}
for $t \geq 0$.
This further yields %for $t \geq 0$
\begin{equation}
\label{DEdecay}
\|\nabla_x E(t) \|_\infty \leq C\left ( 1 + \ln^* \left (\max_{\alpha =1, ..., N} \|\nabla_x f^\alpha(t)\|_\infty \right ) \right ) (1+t)^{-p}
\end{equation}
for $t \geq 0$ and any $p \in (2,3)$ where $p$ can be taken as close to $3$ as desired.

Next, we estimate derivatives of $f^{\alt{\alpha}}$ in order to close the argument.
Taking derivatives in the Vlasov equation, we find 
$$
\begin{gathered}
\partial_{v_k} f^\alpha(t,x,v) = \partial_{v_k} f^\alpha_0(\mcX(0), \mcV(0)) - \int_0^t \partial_{x_k} f^\alpha(s, \mcX(s), \mcV(s)) \ ds,\\
\partial_{x_k} f^\alpha(t,x,v) = \partial_{x_k} f^\alpha_0( \mcX(0), \mcV(0)) - \int_0^t \partial_{x_k}E(s, \mcX(s)) \cdot \nabla_v f^\alpha(s, \mcX(s), \mcV(s)) \ ds.
\end{gathered}
$$
The first equation implies
$$ \| \partial_{v_k} f^\alpha(t) \|_\infty \leq  \| \partial_{v_k} f^\alpha_0 \|_\infty + t  \| \partial_{x_k} f^\alpha(t) \|_\infty \leq (1+t) \mcD(t)$$
for $k = 1, 2, 3$ where
$$\mcD(t) = e^2 + \max_{\alpha =1, ..., N} \|\nabla_v f^\alpha_0 \|_\infty + \sup_{s \in [0,t]} \max_{\alpha =1, ..., N} \| \nabla_x f^\alpha(s) \|_\infty.$$
Using \eqref{DEdecay} and this estimate, the second equation yields for $k = 1, 2, 3$
\begin{eqnarray*}
\| \partial_{x_k} f^\alpha(t) \|_\infty & \leq &  \| \partial_{x_k} f^\alpha_0 \|_\infty + \int_0^t \|\nabla_x E(s) \|_\infty  \| \nabla_v f^\alpha(s) \|_\infty \ ds\\
& \leq & C + C\int_0^t \left ( 1 + \ln^* \left (\max_{\alpha =1, ..., N} \|\nabla_x f^\alpha(s)\|_\infty \right ) \right ) (1+s)^{-p} \| \nabla_v f^\alpha(s) \|_\infty \ ds\\
& \leq & C + C\int_0^t (1+s)^{1-p} \left ( 1 + \ln(\mcD(s)) \right ) \mcD(s) \ ds.
\end{eqnarray*}
Summing over $k$ and taking the maximum over $\alpha$ and the supremum in $t$ ultimately provides
$$\mcD(t) \leq C + C\int_0^t (1+s)^{1-p} \left ( 1 + \ln(\mcD(s)) \right ) \mcD(s) \ ds.$$
Invoking a variant of Gronwall's inequality then yields
$$\mcD(t) \leq C\exp \left ( \exp \left ( \int_0^t (1+s)^{1-p} \ ds \right ) \right ) \leq C$$
for all $t \geq 0$ as $p > 2$. 
As $\mcD(t)$ is bounded, we find
$$\|\nabla_x f^\alpha(t) \|_\infty \lesssim 1 \qquad \mathrm{and} \qquad \|\nabla_v f^\alpha(t) \|_\infty \leq (1+t)\mcD(t) \lesssim t,$$
and the final two conclusions follow.
Additionally, the first conclusion is obtained upon using the bound on $\|\nabla_x f^\alpha(t) \|_\infty$ in \eqref{DEdecayln}.
\end{proof}

%Hence, the decay of the electric field gives rise to decay of its derivatives and estimates on those of the distribution function.
Estimates on the behavior of the derivatives of characteristics follow, as well.

%In order to estimate derivatives of modified characteristics, we first derive estimates on the derivatives of the characteristics of $f$.
\begin{lemma}
\label{LDchar}
For $\alpha = 1, ..., N$, $\tau$ sufficiently large, $t \geq \tau$, and $(x, v) \in \mcU$, we have
\begin{align*}
& \left | \frac{\partial \mcX}{\partial v}(t,\tau, x, v) \right | \lesssim t,  \hspace{1in}
\left | \frac{\partial \mcV}{\partial v}(t,\tau, x, v) \right | \lesssim 1,\\
&\left | \frac{\partial \mcX}{\partial x}(t,\tau, x, v) \right | \lesssim t, \hspace{1in}
\left | \frac{\partial \mcV}{\partial x}(t,\tau, x, v) \right | \lesssim 1.
\end{align*}
\end{lemma}

\begin{proof}
Taking a $v$ derivative in \eqref{char} yields
\begin{equation*}
\left \{
\begin{aligned}
& \frac{\partial \ddot{\mathcal{X}}^\alpha}{\partial v}(t)= \alt{\frac{q_\alpha}{m_\alpha}}\nabla_x E(t, \mcX(t)) \frac{\partial \mcX}{\partial v}(t),\\
& \frac{\partial \mcX}{\partial v}(\tau) = 0 ,\qquad \frac{\partial \dot{\mathcal{X}}^\alpha}{\partial v}(\tau) = \mathbb{I}.
\end{aligned}
\right.
\end{equation*}
Upon integrating, we can rewrite the solution of this differential equation as
$$\frac{\partial \mcX}{\partial v}(t) = (t - \tau)\mathbb{I} + \alt{\frac{q_\alpha}{m_\alpha}}\int_\tau^t (t-s) \nabla_x E(s, \mcX(s)) \frac{\partial \mcX}{\partial v}(s) ds$$
so that by Lemma \ref{DEf} we have for any $p \in (2,3)$
$$\left |\frac{\partial \mcX}{\partial v}(t) \right | \leq  t - \tau  + C\int_\tau^t \frac{t-s}{(1+s)^p} \left | \frac{\partial \mcX}{\partial v}(s) \right | ds.$$
%follows from the above bound.
Now, we fix some $\delta \in [4,6]$ and define
$$T_0(\tau) = \sup \left \{ T \geq \tau : \left | \frac{\partial \mcX}{\partial v}(s) \right | \leq \delta (s - \tau) \ \mathrm{for \ all} \ s \in [\tau, T] \right \}.$$
Note that $T_0 > \tau$ due to the initial conditions. 
%$\frac{\partial \mcX}{\partial w}(\tau) = 0$ and $\frac{\partial \dot{\mcX}}{\partial w}(\tau) = 1$.
Then, estimating for $t \in [\tau, T_0)$, we have
$$\left |\frac{\partial \mcX}{\partial v}(t) \right | \leq  t - \tau  + C\delta \int_\tau^t \frac{(t-s)(s-\tau)}{(1+s)^p} ds.$$
Integrating by parts twice, we find
$$\left |\frac{\partial \mcX}{\partial v}(t) \right | \leq  t - \tau + C\delta (1+\tau)^{2-p} (t-\tau) \alt{\leq (1 + C(1+\tau)^{2-p})(t-\tau)} \leq 2 (t-\tau) \leq \frac{1}{2}\delta (t-\tau)$$
for $\tau$ sufficiently large. Hence, we find $T_0 =\infty$ and the first estimate on $v$-derivatives follows.
The second estimate is directly implied by the first as
$$\frac{\partial \dot{\mathcal{V}}^\alpha}{\partial v}(t)=\alt{\frac{q_\alpha}{m_\alpha}}\nabla_x E(t, \mcX(t)) \frac{\partial \mcX}{\partial v}(t)\alt{.}$$
\alt{Hence,} for $\tau$ sufficiently large
$$\left | \frac{\partial \mcV}{\partial v}(t) \right | \lesssim 1 + \int_\tau^t \alt{(1+s)^{-p}} (s-\tau) \ ds \lesssim 1$$
\alt{for any $p \in (2,3)$ by Lemma \ref{DEf}.}

We estimate $x$ derivatives similarly noting that only the initial conditions change so that 
\begin{equation*}
\left \{
\begin{aligned}
& \frac{\partial \ddot{\mathcal{X}}^\alpha}{\partial x}(t)= \alt{\frac{q_\alpha}{m_\alpha}}\nabla_x E(t, \mcX(t)) \frac{\partial \mcX}{\partial x}(t),\\
& \frac{\partial \mcX}{\partial x}(\tau) = \mathbb{I}  ,\qquad \frac{\partial \dot{\mathcal{X}}^\alpha}{\partial x}(\tau) = 0.
\end{aligned}
\right.
\end{equation*}
As before, we can rewrite the solution as
$$\frac{\partial \mcX}{\partial x}(t) = \mathbb{I} + \alt{\frac{q_\alpha}{m_\alpha}}\int_\tau^t (t-s) \nabla_x E(s, \mcX(s)) \frac{\partial \mcX}{\partial x}(s) ds$$
with
$$\frac{\partial \mcV}{\partial x}(t) = \alt{\frac{q_\alpha}{m_\alpha}} \int_\tau^t \nabla_x E(s, \mcX(s)) \frac{\partial \mcX}{\partial x}(s) ds.$$
Applying analogous tools will yield the stated results. Indeed, fixing $\delta \in [4,6]$ as before and defining
$$T_1(\tau) = \sup \left \{ T \geq \tau : \left | \frac{\partial \mcX}{\partial x}(s) \right | \leq 1 + \delta (1+s) \ \mathrm{for \ all} \ s \in [\tau, T] \right \},$$
we estimate for $t \in [\tau, T_1)$ and integrate by parts to find
\begin{eqnarray*}
\left |\frac{\partial \mcX}{\partial x}(t) \right | &  \leq &  1  + C \int_\tau^t \frac{t-s}{(1+s)^p} \left ( 1 + \delta (1+s) \right )ds\\
& \leq &  1  + C(1+ \tau)^{2-p}(t-\tau) \left ( (1+\tau)^{-1} + \delta \right )\\
& \leq &  1  + C\delta(1+ \tau)^{2-p}(t-\tau)\\
& \leq &  1  + \frac{1}{2} \delta (1 + t)
\end{eqnarray*}
for $\tau$ sufficiently large as $p > 2$. It follows that $T_1 = \infty$, which further implies the linear growth rate of $\frac{\partial \mcX}{\partial x}(t)$. Using this within the equation for $\frac{\partial \mcV}{\partial x}(t)$ then provides the stated uniform-in-time bound as before.

\end{proof}

\subsection{Properties of Translated Distribution Functions}
%\label{sect:g}

Prior to stating the lemmas, we first introduce some notation relating to the translated distribution functions.
As mentioned in the introduction, we let
$$g^\alpha(t,x,v) = f^\alpha(t,x+vt, v)$$
and because the translation alters the spatial characteristics of this system, we further define
\begin{equation}
\label{gcharalt}
\mcY(t,\tau, x,v) = \mcX(t,\tau, x, v) - t \mcV(t, \tau, x, v)
\end{equation}
so that
\begin{equation}
\label{gchar}
 \dot{\mcY}(t) = - \alt{\frac{q_\alpha}{m_\alpha}} t E(t, \mcY(t) + t \mcV(t))
\end{equation}
with $\mcY(\tau) = x - v\tau$. 
In addition, note that $\Vert g^\alpha(t) \Vert_\infty \leq \Vert f_0^\alpha \Vert_\infty$ for all $t \geq 0$.
For $t \geq 0$ define the support of $g^\alpha(t)$ by
$$\alt{\bS_g(t)} = \overline{\left \{ (x,v) \in \R^6 : g^\alpha(t,x,v) \neq 0 \right \}}\alt{.}$$
\alt{We note that the measure of the velocity support of $g^\alpha$ remains uniformly bounded due to Lemma \ref{XVsupp}, i.e.}
$$\sup_{x \in \bfR^3} \left |\bS_v(t, x) \right | \lesssim 1.$$

As our approach relies heavily upon the growth of the spatial support and velocity derivatives of $g^\alpha$, 
we further define 
\alt{the set
$$\bS_x(t,v) = \overline{\left \{x \in \bfR^3 : g^\alpha(t,x,v) \neq 0 \right \}}$$
for any fixed $v \in \bfR^3$ and} 
the useful quantities
$$\alt{\mu(t) = \max \left \{ 1, \max_{\alpha = 1, ..., N} \sup_{v\in \bfR^3} \left | \alt{\bS_x(t,v)} \right | \right \}}$$
and
$$\mcG(t) =  \max \left \{ 1,\max_{\alpha = 1, ..., N}  \| \nabla_v g^\alpha (t) \|_\infty \right \}.$$
With suitable decay of the field and its derivatives, we now study the behavior of the translated characteristics and establish a measure-preserving property as for the characteristics of \eqref{VP}.

\begin{lemma}
\label{MPP}
For every $\alpha = 1, ..., N$ and $t \geq 0$, the mapping $(x,v) \mapsto (\mcY(t), \mcV(t))$ defined by the characteristics satisfying
\begin{equation}
\label{gcharsys}
\left \{
\begin{aligned}
&\dot{\mcY}(t)= - \alt{\frac{q_\alpha}{m_\alpha}} t E(t, \mcY(t) + t \mcV(t))\\
&\dot{\mcV}(t)= \alt{\frac{q_\alpha}{m_\alpha}} E(t, \mcY(t) + t \mcV(t))
\end{aligned}
\right.
\end{equation}
with initial conditions
$\mcY(\tau, \tau, x, v) = x-v\tau$ and
$\mcV(\tau, \tau, x, v) = v,$
is a measure-preserving diffeomorphism. 
Hence, $| \alt{\bS_g(t)} | = | \alt{\bS_g(0)} |$
for every $\alpha = 1, ..., N$ and $t \geq 0$.
\end{lemma}

\begin{proof}
\alt{Because the system \eqref{VPg} is generated by a divergence-free vector field, namely it is of the form
$$\partial_t g^\alpha + \Phi(t,x,v)  \cdot \nabla_{(x,v)} g^\alpha(t,x,v)= 0$$
where
$$\Phi(t,x,v) = \frac{q_\alpha}{m_\alpha}\left (-t E(t,x+vt), E(t,x+vt) \right )$$
satisfies $\nabla_{(x,v)} \cdot \Phi = 0$, 
the measure-preserving property of characteristics in phase space is merely a consequence of the classical Liouville Theorem \cite[pp. 63-65]{Tuckerman} from statistical mechanics.
Additionally, the proof is analogous to that for the characteristics of the untranslated Hamiltonian system \eqref{VP}. We direct the reader to  \cite[pp. 118-120]{Glassey} for further details.}

%The proof is analogous to that of \cite{Glassey} for the characteristics of the untranslated Hamiltonian system, but we include it for completeness.
%We begin with \eqref{gcharsys} and formulate the variational equations by replacing $\mcY \to \mcY + \epsilon y$ and $\mcV \to \mcV + \epsilon w$
%so that
%$$
%\left \{
%\begin{aligned}
%&\frac{d}{ds} \left ( \mcY + \epsilon y \right ) = -\alt{\frac{q_\alpha}{m_\alpha}} sE(s, \mcY + s \mcV + \epsilon (y + ws) )\\
%&\frac{d}{ds} \left ( \mcV + \epsilon w \right ) = \alt{\frac{q_\alpha}{m_\alpha}}E(s, \mcY + s \mcV + \epsilon (y + ws) ).
%\end{aligned}
%\right.
%$$
%Then, taking a derivative with respect to $\epsilon$ and evaluating at $\epsilon = 0$ yields
%\begin{equation}
%\label{charvar}
%\left \{
%\begin{aligned}
%& \dot{y} = -\alt{\frac{q_\alpha}{m_\alpha}} sE_X(s, \mcY + s \mcV ) (y+ws) \\
%& \dot{w} = \alt{\frac{q_\alpha}{m_\alpha}}E_X(s, \mcY + s \mcV ) (y+ws).
%\end{aligned}
%\right.
%\end{equation}
%Letting $u = [y, w]^T$, we can rewrite this linear matrix system as $\dot{u} = A(s)u$
%where
%$$A(s) =\alt{\frac{q_\alpha}{m_\alpha}}
%\begin{bmatrix}
%-sE_X(s, \mcX(s)) & - s^2 E_X(s, \mcX(s))\\
% E_X(s, \mcX(s)) & s E_X(s, \mcX(s))
%\end{bmatrix}
%$$
%using \eqref{gcharalt}.
%As for the variational system launched by \eqref{char}, we see that $\mbox{tr}(A(s)) = 0$ for all $s \geq 0$, and thus the fundamental matrix solution of \eqref{charvar}, say $U(t)$, satisfies
%$$\det(U(t)) = \det(U(0)) \exp \left ( \int_0^t \mathrm{tr}(A(\tau)) d\tau \right ) = \det(U(0)) = 1$$ 
%for all $t \geq 0$. Hence, the mapping preserves measure.
\end{proof}

Next, we establish a bound on the growth of the spatial support of the translated distribution functions, which will be instrumental within subsequent sections.
\begin{lemma}
\label{Xsupp}
For every $\alpha = 1, ..., N$,  $\tau \geq 0$, and $(x,v) \in \mcU$ we have
$$\left | \mcY(t, \tau, x, v)  - (x -v\tau) \right | \lesssim \int_\tau^t s \| E(s) \|_\infty \ ds$$
and 
$$\alt{\mu(t)} \lesssim \left ( 1 + \int_0^t s \| E(s) \|_\infty \ ds \right )^3.$$
\alt{In particular, Lemma \ref{XVsupp} implies
$$\left | \mcY(t, \tau, x, v) \right | \lesssim \ln(t) \qquad \mathrm{and} \qquad 
\mu(t) \lesssim \ln^3(t).$$}
\end{lemma}

\begin{proof}
Using \eqref{gchar} we immediately find
$$\left | \dot{\mcY}(t) \right |  \leq Ct \Vert E(t) \Vert_\infty,$$
and thus
$$\left | \mcY(t, \tau, x, v)  - (x - v \tau) \right | \leq \int_\tau^t \left |  \dot{\mcY}(s, \tau, x, v)  \right | \ ds \lesssim \int_\tau^t s \Vert E(s) \Vert_\infty \ ds.$$
for fixed $\tau \geq 0$ and $(x,v) \in \mcU$.
\alt{Furthermore, this implies
$$\left | \mcY(t, 0, x, v) \right | \lesssim  |x| + \int_\tau^t s \Vert E(s) \Vert_\infty \ ds \lesssim 1 + \int_\tau^t s \Vert E(s) \Vert_\infty \ ds.$$
for $(x,v) \in \bS_g(0)$.
The estimate on the spatial support then follows as
$$\mu(t) \lesssim \left (\max_{\alpha = 1, ..., N}\sup_{(x,v) \in \bS_g(0)} \left | \mcY(t, 0, x, v) \right | \right )^3 \lesssim \left ( 1 + \int_\tau^t s \Vert E(s) \Vert_\infty \ ds \right)^3.$$}
% immediately from this uniform (in $v$ and $\alpha$) bound and the compact support of $g^\alpha(t)$ at any finite time $t$.%, while the latter results follow from applying \eqref{Edecay}.
\end{proof}

Additionally, we estimate the derivatives of these spatial characteristics using the behavior of the untranslated characteristics established in Lemma \ref{LDchar}.

\begin{lemma}
\label{dY}
For every $\alpha = 1, ..., N$, $\tau$ sufficiently large, $t \geq \tau$, and $(x, v) \in \mcU$, we have
$$\left | \frac{\partial \mcY}{\partial x}(t,\tau, x, v) - \mathbb{I} \right | \lesssim \int_\tau^t s^2 \| \nabla_x E(s) \|_\infty ds
\quad \mathrm{and} \quad
\left | \frac{\partial \mcY}{\partial v}(t,\tau, x, v)   + \tau \mathbb{I} \right | \lesssim \int_\tau^t s^2 \| \nabla_x E(s) \|_\infty ds.$$
%which further yields
%$$\left | \frac{\partial \mcY}{\partial x}(t,\tau, x, v) \right | \lesssim 1 \qquad \mathrm{and} \qquad
%\left | \frac{\partial \mcY}{\partial v}(t,\tau, x, v)   + \tau \mathbb{I} \right | \lesssim \ln^2(t).$$
\end{lemma}

\begin{proof}
Using \eqref{gcharalt} and \eqref{gchar} we write 
$$\mcY(t,\tau,x,v) = x - v\tau - \alt{\frac{q_\alpha}{m_\alpha}} \int_\tau^t s E(s, \mcX(s)) \ ds,$$
and taking an $x$ derivative yields
$$\frac{\partial \mcY}{\partial x}(t)= \mathbb{I} - \alt{\frac{q_\alpha}{m_\alpha}} \int_\tau^t s \nabla_x E(s, \mcX(s)) \frac{\partial \mcX}{\partial x}(s) \ ds.$$
Using Lemma \ref{LDchar}, we find
$$\left | \frac{\partial \mcY}{\partial x}(t) - \mathbb{I} \right | \lesssim \int_\tau^t s^2 \| \nabla_x E(s) \|_\infty ds$$ %\leq C\int_\tau^t s^{-2}\ln(s) \ ds \leq C$$
for $\tau$ sufficiently large.
Proceeding similarly for $v$-derivatives, we have
$$\frac{\partial \mcY}{\partial v}(t)= -\tau \mathbb{I} - \alt{\frac{q_\alpha}{m_\alpha}} \int_\tau^t s \nabla_x E(s, \mcX(s)) \frac{\partial \mcX}{\partial v}(s) \ ds.$$
Again using Lemma \ref{LDchar}, for $\tau$ sufficiently large we find 
$$\left | \frac{\partial \mcY}{\partial v}(t) +\tau \mathbb{I} \right | \lesssim \int_\tau^t s^2 \| \nabla_x E(s) \|_\infty ds.$$ % \leq C\int_\tau^t s^{-1}\ln(s) \ ds \leq C\ln^2(t)$$
\end{proof}

Finally, we show that velocity derivatives of $g^\alpha$ grow more slowly than those of $f^\alpha$, which are established by Lemma \ref{DEf}. 

\begin{lemma}
\label{Dvg}
We have
$$\mcG(t) \lesssim 1 + \int_0^t s^2 \|\nabla_x E(s) \|_\infty \ ds,$$
\alt{and thus by Lemma \ref{DEf}
$$\mcG(t) \lesssim \ln^2(t).$$}
\end{lemma}

\begin{proof}
%Integrating the Vlasov equation for $g^\alpha$ or $f^\alpha$ in $x$, we note that for every $t \geq 0$ and $v \in \bfR^3$
%$$\int g^\alpha(t,x,v) \ dx = \int f^\alpha(t, x+vt, v) \ dx = \int f^\alpha(t, x, v) \ dx = F^\alpha(t,v).$$
%
To establish the result we estimate
$$\partial_{v_k} g^\alpha(t,x,v) = \left ( t \partial_{x_k} f^\alpha + \partial_{v_k} f^\alpha \right )(t,x+vt, v).$$
Applying the Vlasov operator to the untranslated version of this quantity yields
$$\frac{d}{dt} \biggl ( t\partial_{x_k} f^\alpha(t, \mcX(t), \mcV(t)) + \partial_{v_k} f^\alpha(t, \mcX(t), \mcV(t))  \biggr )
= - \alt{\frac{q_\alpha}{m_\alpha}} t\partial_{x_k} E(t,\mcX(t)) \cdot \nabla_v f^\alpha(t, \mcX(t), \mcV(t)),$$
and inverting gives
$$\left (t\partial_{x_k} f^\alpha + \partial_{v_k} f^\alpha \right )(t,x,v) = \partial_{v_k} f^\alpha_0(\mcX(0),\mcV(0)) - \alt{\frac{q_\alpha}{m_\alpha}} \int_0^t s \partial_{x_k} E(s,\mcX(s)) \cdot \nabla_v f^\alpha(s, \mcX(s), \mcV(s))  \ ds$$
for $k = 1, 2, 3$.
From the estimates of Lemma \ref{DEf}, this implies
\begin{eqnarray*}
\left | t\partial_{x_k} f^\alpha(t,x,v) + \partial_{v_k} f^\alpha(t,x,v) \right | & \leq & \| \partial_{v_k} f^\alpha_0\|_\infty + \int_0^t s \|\nabla_x E(s) \|_\infty \|\nabla_v f^\alpha(s)\|_\infty \ ds \\
& \lesssim & 1 + \int_0^t s^2 \|\nabla_x E(s) \|_\infty ds
%& \leq & C + C\int_0^t (1 + \ln(1+s))(1+s)^{-1} \ ds\\
%& \leq & C(1 + \ln(1+t))^2
\end{eqnarray*}
for $k = 1, 2, 3$. Hence, we find
$$ \| \nabla_v g^\alpha(t) \|_\infty \lesssim 1 + \int_0^t s^2 \|\nabla_x E(s) \|_\infty ds$$
%$$ \| \nabla_v g^\alpha(t) \|_\infty \leq C(1 + \ln(1+t))^2$$
for all $\alpha = 1, ..., N$, which completes the lemma upon taking the maximum in $\alpha$.
\end{proof}

\section{Velocity Limits and Convergence of the Spatial Average}
\label{sect:vellim}

Because the field decays rapidly in time, we can immediately establish the limiting behavior of the velocity characteristics. 
%Furthermore, an asymptotic approximation for the behavior of spatial characteristics follows.

\begin{lemma}
\label{L6}
For any $\tau \geq 0$, $(x, v) \in \mcU$, and $\alpha = 1, ..., N$ the limiting velocities $\mcV_\infty$ defined by
$$\mcV_\infty(\tau, x, v) :=  \lim_{t \to \infty} \mcV(t, \tau, x, v) = v + \alt{\frac{q_\alpha}{m_\alpha}}\int_\tau^\infty E(s, \mcX(s, \tau, x, v)) ds.$$ 
exist, and are $C^2$, bounded, and invariant under the characteristic flow\alt{, namely $\mcV_\infty$ satisfies
$$\mcV_\infty(t, \mcX(t, \tau, x,v), \mcV(t, \tau, x,v)) = \mcV_\infty(\tau, x,v)$$
for any $t \geq 0$.}
Finally, we have the convergence estimate
$$|\mcV(t, \tau, x, v) - \mcV_\infty(\tau, x, v) | \lesssim \int_t^\infty \|E(s) \|_\infty \ ds,$$
which further yields
\begin{equation}
\label{Winfest}
|\mcV(t, \tau, x, v) - \mcV_\infty(\tau, x, v) | \lesssim t^{-1}.
\end{equation}
%and
%$$\left |\mcV_\infty(\tau, x, v) - v \right | \leq C(1+\tau)^{-1}$$
%for $\tau \geq 0$.
\end{lemma}
\begin{proof}
For any $(x,v) \in \mcU$, we find from the characteristic equations
$$\mcV(t,\tau, x, v) = v + \alt{\frac{q_\alpha}{m_\alpha}} \int_\tau^t E(s, \mcX(s, \tau, x, v)) ds.$$
Thus, define
$$\mcV_\infty(\tau, x, v) = v+ \alt{\frac{q_\alpha}{m_\alpha}} \int_\tau^\infty E(s, \mcX(s, \tau, x, v)) ds$$
for every $\tau \geq 0$ and $(x,v) \in \mcU$.
From Lemma \ref{XVsupp} %Lemmas \ref{L2} and \ref{L3}, 
we find 
$$|\mcV_\infty(\tau)| \alt{\lesssim} |v| + \int_\tau^\infty \Vert E(s) \Vert_\infty \ ds \alt{\lesssim} 1 + \tau^{-1} \alt{\lesssim} 1.$$
Additionally, the convergence estimate
$$| \mcV(t) -\mcV_\infty| \alt{\lesssim} \int_t^\infty \| E(s) \|_\infty ds \lesssim t^{-1}.$$ 
%for $t\geq 0$ and
%$$\left |\mcV_\infty(\tau) - v \right |  = \int_\tau^\infty E(s, \mcX(s, \tau, x, v)) ds \leq C(1+\tau)^{-1}.$$
%for $\tau \geq 0$
follows immediately.
Thus, the limit exists and is uniformly bounded for every $\tau \geq 0$ and $(x,v) \in \mcU$.
Finally, $\mcV_\infty$ is $C^2$ due to the regularity of the electric field and spatial characteristics \alt{and invariant under the flow due to the time-reversibility of \eqref{char}}. 
%Finally, we show that the limiting velocity of a particle is invariant under the characteristic flow. In particular, for every $\tau, t \geq 0$ and $(x, v) \in \mcU$, we use the notation $\mcX(t) = \mcX(t, \tau, x, v)$ and $\mcV(t) = \mcV(t, \tau, x, v)$ and note that
%$$ \mcX(s, t, \mcX(t), \mcV(t)) = \mcX(s, \tau, x, v) = \mcX(s).$$
%With this, we have
%\begin{eqnarray*}
%\mcV_\infty(t, \mcX(t), \mcV(t)) & = & \mcV(t) + \int_t^\infty E(s, \mcX(s, \tau, \mcX(t), \mcV(t))) ds\\
%& = & \mcV(t) + \int_t^\tau E(s, \mcX(s)) ds + \int_\tau^\infty E(s, \mcX(s)) ds\\
%& = & \mcV(\tau) + \int_\tau^\infty E(s, \mcX(s)) ds\\
%& = & v +  \int_\tau^\infty E(s, \mcX(s)) ds\\
%& = & \mcV_\infty(\tau, x, v),
%\end{eqnarray*}
%and the proof is complete.
 \end{proof}

%With the limiting velocities established, we can precisely determine the asymptotic behavior of the spatial characteristics, as well.
%
%\begin{lemma}
%\label{L8}
%Let $\mcU \subset \bfR^6$ be compact. For any $\tau \geq 0$ and $(x, v) \in \mcU$, we find 
%$$\lim_{t \to \infty} \frac{\mcX(t,\tau,x, v) - x}{t - \tau} = \mcV_\infty(\tau, x, v).$$
%In particular, for $t \geq \tau \geq 0$ and $(x, v) \in \mcU$, we have the asymptotic representation
%$$\mcX(t, \tau, x, v) =  x + \mcV_\infty(\tau, x, v) (t-\tau) + \mathcal{O}\left (\ln \left (\frac{1+t}{1 + \tau} \right) \right). $$
%\end{lemma}
%
%\begin{proof}
%Using the convergence estimate in Lemma \ref{L6}, we have
%\begin{eqnarray*}
%\left | \frac{\mcX(t,\tau, x, v) - \left [x +  \mcV_\infty(\tau, x, v) (t-\tau)\right ]}{t- \tau} \right | & \leq & \frac{1}{t - \tau} \int_\tau^t \left | \mcV(s, \tau, x, v) -  \mcV_\infty(\tau, x, v) \right | \ ds\\
%& \leq & \frac{C}{t - \tau} \int_\tau^t  (1+s)^{-1} ds\\
%& = & \frac{C}{t - \tau}\ln\left ( \frac{1+t}{1+\tau} \right)
%\end{eqnarray*}
%for $0 \leq \tau < t$.
%Taking $t \to \infty$ produces the limiting result and multiplying by $t - \tau$ yields the stated asymptotic estimate.
% \end{proof}

Next, we control the growth in derivatives of these limits and show that they are invertible functions of $v$ for sufficiently large $\tau$. This property will be useful later to perform a change of variables and establish limits of spatial averages.

\begin{lemma}
\label{Vinfinvert}
For every $\alpha = 1, ..., N$ and $(x,v) \in \mcU$ we have
%$$\left |\frac{\partial \mcV_\infty}{\partial x}(\tau, x, v) \right | \lesssim \int_\tau^\infty \Vert \nabla_x E(s) \Vert_\infty \ ds \quad \mathrm{and} \quad \left |\frac{\partial \mcV_\infty}{\partial v}(\tau, x, v) - \mathbb{I} \right | \lesssim \int_\tau^\infty s \| \nabla_x E(s) \|_\infty ds,$$
%which yields
$$\left |\frac{\partial \mcV_\infty}{\partial x}(\tau, x, v) \right | \lesssim \tau^{-2}\ln(\tau) \qquad \mathrm{and} \qquad \left |\frac{\partial \mcV_\infty}{\partial v}(\tau, x, v) - \mathbb{I} \right | \lesssim \tau^{-1}\ln(\tau).$$
Thus, there is $T_1 > 0$ such that for all $\tau \geq T_1$ and $(x, v) \in \mcU$, we have
$$\left |\det \left (\frac{\partial \mcV_\infty}{\partial v}(\tau, x, v) \right ) \right | \geq \frac{1}{2}.$$
Consequently, for $\tau \geq T_1$ and $(x, v) \in \mcU$, the $C^2$ mapping $v \mapsto \mcV_\infty(\tau, x, v)$ is injective and invertible.
\end{lemma}

\begin{proof}
First, note that the limiting velocities given by Lemma \ref{L6} satisfy
$$ \frac{\partial \mcV_\infty}{\partial v}(\tau, x, v) = \mathbb{I} + \alt{\frac{q_\alpha}{m_\alpha}}\int_\tau^\infty \nabla_x E(s, \mcX(s))  \frac{\partial \mcX}{\partial v}(s) \ ds. $$
Due to Lemmas \ref{DEf} and \ref{LDchar}, we find for $(x, v) \in \mcU$ %$\tau \geq 0$ and
$$\left | \frac{\partial \mcV_\infty}{\partial v}(\tau) - \mathbb{I} \right | \lesssim \int_\tau^\infty s \| \nabla_x E(s) \|_\infty ds \lesssim \tau^{-1}\ln(\tau).$$
Therefore, by the continuity of the mapping $A \mapsto \det(A)$, there is $T_1 > 0$ such that for all $\tau \geq T_1$ and $(x, v) \in \mcU$, we have
$$ \left | \det \left (\frac{\partial \mcV_\infty}{\partial v}(\tau, x, v)  \right )\right | \geq \frac{1}{2}.$$
%
%Note that $\mcV_\infty(\tau, x, v)$ is actually $C^2$, as $E(t,x)$ is sufficiently smooth due to the regularity of classical solutions. 
%For brevity, we omit the details.
Finally, we obtain the estimate on $x$-derivatives in the same manner so that
$$ \frac{\partial \mcV_\infty}{\partial x}(\tau, x, v) = \alt{\frac{q_\alpha}{m_\alpha}}\int_\tau^\infty \nabla_x E(s, \mcX(s))  \frac{\partial \mcX}{\partial x}(s) \ ds, $$
and using Lemma \ref{LDchar}, this implies
$$ \left | \frac{\partial \mcV_\infty}{\partial x}(\tau, x, v) \right | \lesssim \int_\tau^\infty \Vert \nabla_x E(s) \Vert_\infty \ ds \lesssim \tau^{\alt{-2}}\ln(\tau).$$
\end{proof}

Next, we define the collection of all limiting velocities on $\alt{\bS_f(t)}$, which will serve as the support of the limiting spatial average.
First notice that due to its invariance under the flow defined by \eqref{char}, we have
$$\left \{ \mcV_\infty(\tau, x, v) : (x, v) \in \alt{\bS_f(\tau)} \right \} = \left \{ \mcV_\infty(0, x, v) : (x, v) \in \alt{\bS_f(0)} \right \}$$
for all $\tau \geq 0$. Hence, for any $\alpha = 1, ..., N$ define
$$\Omv := \left \{ \mcV_\infty(0, x, v) : (x, v) \in \alt{\bS_f(0)}\right \}.$$
As $\mcV_\infty (0,x, v)$ is continuous due to Lemma \ref{L6}, its range $\Om_v^\alpha$ on the compact set $\alt{\bS_f(0)}$ is also compact.
With the behavior of velocity characteristics well-understood as $t \to \infty$, we wish to understand the distribution of particle velocities in this same limit, and this will be given by the limiting behavior of spatial averages.
In particular, \alt{we remind the reader that} applying 
Lemma \ref{XVsupp} and the field derivative estimate of Lemma \ref{DEf} to Lemmas \ref{Xsupp} and \ref{Dvg} implies
\begin{equation}
\label{Sx}
\alt{\mu(t)} \lesssim \ln^3(t)
\end{equation}
and
\begin{equation}
\label{G}
\mcG(t) \lesssim \ln^2(t),
\end{equation}
which will be significant to establishing the uniform convergence of the spatial average.

\begin{lemma}
\label{Funif}
For every $\alpha = 1, ..., N$ there exists $F^\alpha_\infty \in C_c^2(\bfR^3)$ with $\mathrm{supp}(F^\alpha_\infty) = \Om^\alpha_v$ such that
$$F^\alpha(t,v) = \int f^\alpha(t,x, v) \ dx$$
satisfies $F^\alpha(t,v) \to F^\alpha_\infty(v)$ uniformly as $t \to \infty$
with
$$\| F^\alpha(t) - F^\alpha_\infty \|_\infty \lesssim \int_t^\infty \left (s \| \rho(s)\|_\infty  +  \|E(s)\|_\infty \alt{\mu(s)} \mcG(s) \right ) ds,$$
and thus
$$\| F^\alpha(t) - F^\alpha_\infty \|_\infty \lesssim t^{-1}\ln^5(t).$$
%$$\lim_{t \to \infty} \int \psi(v) F^\alpha(t,v) dv= \int \psi(v) F^\alpha_\infty(v) dv$$
%for every $\psi \in C_b \left (\bfR^3 \right )$.
In particular, the limit preserves \alt{particle number, velocity}, and energy, i.e.
$$\int F^\alpha_\infty(v) \ dv= \mcM^\alpha, \qquad \int v F^\alpha_\infty(v) \ dv= \mcJ^\alpha, \qquad \mathrm{and} \qquad 
\sum_{\alpha = 1}^N \int \frac{1}{2} \alt{m_\alpha} |v|^2 F^\alpha_\infty(v) \ dv= \mcEVP.$$
\end{lemma}

\begin{proof}
Upon integrating the Vlasov equation of \eqref{VPg} in $x$ and integrating by parts, we find 
\begin{eqnarray*}
\left | \partial_t \int g^\alpha(t,x,v) \ dx \right | & = & \left | \alt{\frac{q_\alpha}{m_\alpha}} \int E(t, x+vt) \cdot (t \nabla_x - \nabla_v) g^\alpha(t,x,v) \ dx \right |\\
& = & \left | \alt{\frac{q_\alpha}{m_\alpha}} t  \int \rho(t, x+vt) g^\alpha(t,x,v) \ dx + \alt{\frac{q_\alpha}{m_\alpha}} \int E(t, x+vt) \cdot \nabla_v g^\alpha(t,x,v) \ dx \right |\\
& \lesssim & t  \Vert \rho(t) \Vert_\infty F^\alpha(t,v) + \Vert E(t) \Vert_\infty \alt{\mu(t)} \Vert \nabla_v g^\alpha(t)\Vert_\infty.
\end{eqnarray*}
Thus, we use Lemma \ref{XVsupp}, \eqref{Sx}, and \eqref{G} to find
$$\left | \partial_t F^\alpha(t,v) \right | \lesssim t^{-2} F^\alpha(t,v) + t^{-2}\ln^5(t).$$
As $F^\alpha(0) \in L^\infty(\bfR^3)$ and the latter term above is integrable in time, we find
$$F^\alpha(t,v) \leq F^\alpha(0,v) + \int_0^t \left | \partial_t F^\alpha(s,v) \right | \ ds \leq C + C \int_0^t (1+s)^{-2} F^\alpha(s, v) \ ds,$$
and after taking the supremum and invoking Gronwall's inequality, this yields
\begin{equation}
\label{FLinfty}
\| F^\alpha(t) \|_\infty \leq C \exp \left (C\int_0^t (1+s)^{-2} \ ds \right ) \leq C
\end{equation}
for all $t \geq 0$.
Returning to the estimate of $\partial_t F^\alpha$, we use the uniform bound on $\|F^\alpha(t)\|_\infty$ to find
$$\left | \partial_t F^\alpha(t,v) \right |  \lesssim t \| \rho(t)\|_\infty  +  \|E(t)\|_\infty \alt{\mu(t)} \mcG(t),$$
and thus
$$\left | \partial_t F^\alpha(t,v) \right |  \lesssim t^{-2}\ln^5(t),$$
which implies that $ \| \partial_t F^\alpha(t)\|_\infty$ is integrable.
This bound then establishes the estimate for $s \geq t$
$$\Vert F^\alpha(t) - F^\alpha(s) \Vert_\infty = \left \Vert \int_s^t \partial_t F^\alpha(\tau) \ d\tau \right \Vert_\infty
\leq \int_t^s \Vert \partial_t F^\alpha(\tau) \Vert_\infty \ d\tau
 \lesssim t^{-1}\ln^5(t),$$
and taking $s \to \infty$ establishes the limit. More precisely, as $F^\alpha(t,v)$ is continuous and the limit is uniform, there is $\alt{F^\alpha_\infty} \in C(\bfR^3)$ such that
$$ \| F^\alpha(t) - \alt{F^\alpha_\infty} \|_\infty  \lesssim t^{-1}\ln^5(t).$$
\alt{Next, we verify the properties of the limiting spatial average.
Due to the uniform convergence, we further conclude weak convergence of $F^\alpha(t,v)$ as a measure, namely
\begin{equation}
\label{Weaklimit}
\lim_{t \to \infty} \int \psi(v) F^\alpha(t,v) dv= \int \psi(v) F^\alpha_\infty(v) dv
\end{equation}
for every $\psi \in C_b \left (\bfR^d \right )$.}
\alt{In this direction, l}et $\psi  \in C_b(\bfR^d)$ be given and fix any $T \geq T_1$ from Lemma \ref{Vinfinvert} and $\alpha = 1, ..., N$.
Then, we apply the measure-preserving change of variables $(\tilde{x}, \tilde{v}) = ( \mcX(T, t, x, v), \mcV(T, t, x, v))$, so that 
\begin{eqnarray*}
\lim_{t \to \infty} \int \psi(v) F^\alpha(t,v) \ dv
& = & \lim_{t \to \infty} \iint \psi(v) \ f^\alpha(t,x,v) \ dv dx\\
& = & \lim_{t \to \infty} \iint\limits_{\alt{\bS_f(t)}} \psi(v) \ f^\alpha(T, \mcX(T,t,x,v), \mcV(T,t,x,v)) \ dv dx\\
& = & \lim_{t \to \infty}\iint\limits_{\alt{\bS_f(T)}} \psi(\mcV(t,T,\tilde{x}, \tilde{v})) \ f^\alpha(T, \tilde{x}, \tilde{v}) \ d\tilde{v} d\tilde{x}\\
& = & \iint\limits_{\alt{\bS_f(T)}} \psi(\mcV_\infty(T,\tilde{x}, \tilde{v})) f^\alpha(T, \tilde{x}, \tilde{v}) \ d\tilde{v} d\tilde{x}
\end{eqnarray*}
by Lebesgue's Dominated Convergence Theorem. Now, by Lemma \ref{Vinfinvert} for any $(\tilde{x}, \tilde{v}) \in \alt{\bS_f(T)}$, the mapping $\tilde{v} \mapsto \mcV_\infty(T, \tilde{x}, \tilde{v})$ is $C^2$ with $$\left |\det \left (\frac{\partial \mcV_\infty}{\partial v}(T, \tilde{x}, \tilde{v}) \right )\right | \geq \frac{1}{2},$$ and thus bijective from $\alt{\bS_v(T,\tilde{x})}$ to $\Om_v$. 
Hence, letting $u = \mcV_\infty(T, \tilde{x}, \tilde{v})$, we perform a change of variables and drop the tilde notation to find
$$\lim_{t \to \infty} \int \psi(v) F^\alpha(t,v) dv
= \iint \psi(u) f^\alpha(T, x, V(u)) \frac{\chfn_{\Om^\alpha_v}(u)}{\left | \det \left (\frac{\partial \mcV_\infty}{\partial v}(T, x, V(u)) \right ) \right |} \ du dx$$
where for fixed $x$, the $C^2$ function $V: \Omv \to \alt{\bS_v(T, x)}$ is given by
$$V(u) = \left (\mcV_\infty \right)^{-1}(T,x,u).$$
\alt{Therefore, by uniqueness of the weak limit we find from \eqref{Weaklimit}}
$$F^\alpha_\infty(u) = \int f^\alpha(T, x, V(u)) \chfn_{\Om^\alpha_v}(u) \left | \det \left (\frac{\partial \mcV_\infty}{\partial v}(T, x, V(u)) \right) \right |^{-1} \ dx.$$
%and notice
%$$| F^\alpha_\infty(u) | \leq 2\|f_0\|_\infty \sup_{v \in \bfR^3} \left |\bS_x(T,v)\right | \leq C.$$
\alt{for any $u \in \bfR^d$, and thus $F^\alpha_\infty \in C_c^2(\bfR^d)$.}
Notice further that due to the compact support and regularity of $F^\alpha(t)$ and $F^\alpha_\infty$, it is sufficient to take $\psi \in C(\mcU)$, where $\mcU \in \bfR^3$ is compact, for this equality to hold.

Next, we show that the conservation laws are maintained in the limit.
In particular, \alt{particle number and velocity} conservation for each species is obtained by merely choosing $\psi(v) = 1$ and $\psi(v) = v$, respectively, within \eqref{Weaklimit} and using the time-independence of these quantities.
Similarly, to establish the energy identity we first note that the potential energy is known \emph{a priori} to decay \cite{IR}, namely
$$\Vert E(t) \Vert_2 \lesssim t^{-1/2}.$$
With this, we use energy conservation %\eqref{EL2}, %Theorem \ref{optimaldecay} with $p = 2$, 
and \eqref{Weaklimit} with $\psi(v) = \frac{1}{2}|v|^2$ for each $\alpha = 1, .., N$ to find
\begin{eqnarray*}
\mcEVP & = & \lim_{t \to \infty} \left (\sum_{\alpha = 1}^N \iint \frac{1}{2} \alt{m_\alpha} |v |^2 f^\alpha(t, x,v) \ dv dx + \frac{1}{2} \int |E(t,x)|^2 \ dx \right)\\
& = & \lim_{t \to \infty} \sum_{\alpha = 1}^N \iint \frac{1}{2} \alt{m_\alpha} |v |^2 f^\alpha(t, x,v) \ dv dx\\
& = & \sum_{\alpha = 1}^N \lim_{t \to \infty}  \int \frac{1}{2} \alt{m_\alpha} |v|^2 F^\alpha(t,v)  dv\\
& = & \sum_{\alpha = 1}^N \int \frac{1}{2} \alt{m_\alpha} |v|^2 F^\alpha_\infty(v) \   dv. 
\end{eqnarray*}
\end{proof}

\begin{corollary}
\label{P}
Define the net velocity density
$$\mcP(t,v) = \sum_{\alpha = 1}^N \alt{q_\alpha} F^\alpha(t,v)$$
and $\Om_v = \bigcup\limits_{\alpha=1}^N \Om^\alpha_v$.
Then, the limiting density
$$\mcP_\infty(v) = \sum_{\alpha=1}^N \alt{q_\alpha} F_\infty^\alpha(v)$$
satisfies $\alt{\mcP_\infty} \in C_c^2(\bfR^3)$ with
$\mathrm{supp}(\mcP_\infty) = \Om_v$
and
$\mcP(t,v) \to \mcP_\infty(v)$ uniformly as $t \to \infty$ with
$$\| \mcP(t) - \mcP_\infty \|_\infty \lesssim \int_t^\infty \left [s \| \rho(s)\|_\infty  +  \|E(s)\|_\infty \alt{\mu(s)}\mcG(s) \right ] ds,$$
and thus
$$\| \mcP(t) - \mcP_\infty \|_\infty \lesssim t^{-1}\ln^5(t).$$
In particular, we have
$\int \mcP_\infty(v) \ dv= \mcM.$
%$$\int_{\Om_v} \mcP_\infty(v) \ dv= \mcM \qquad \mathrm{and} \qquad  \int_{\Om_v} v \mcP_\infty(v) \ dv= \mcJ.$$
\end{corollary}

\section{Convergence of the Field and Macroscopic Densities}
\label{sect:fieldconv}

Now that the spatial averages are known to converge, we next establish the precise asymptotic profile of the field, its derivatives, and the charge and current densities.
Given the limiting density $\mcP_\infty(v)$ established by Corollary \ref{P}, we define its induced electric field by
$$E_\infty(v) = \nabla_v (\Delta_v)^{-1} \mcP_\infty(v) = \alt{\frac{1}{4\pi}}\int \frac{\xi}{|\xi|^3} \mcP_\infty(v -\xi) \ d\xi$$
for every $v \in \bfR^3$.
\alt{Such an induced field was first identified in \cite{Ionescu} to identify the correction exhibited by the behavior of translated spatial trajectories as $t \to \infty$.}
To ensure the necessary regularity of the limiting field we note that due to Lemma \ref{LE} $\|\partial_{x_j}^k E_\infty\|_\infty$ is finite for $j=1, 2, 3$ and $k=0, 1, 2$, as $\partial_{x_j}^k \mcP_\infty \in C_c(\bfR^3) \subseteq L^1(\bfR^3) \cap L^\infty(\bfR^3)$.
With this, we establish a refined estimate of the electric field.

\begin{lemma}
\label{LField}
We have
$$\sup_{x \in \bfR^3} \left | t^2 E(t,x) - E_\infty\left (\frac{x}{t} \right) \right | \lesssim \| \mcP(t) - \mcP_\infty\|_\infty + t^{-1}  \alt{\mu(t)}\mcG(t),$$
Alternatively, 
assume that there is $\beta : [0,\infty) \to [0,\infty)$ such that $| x | \leq \beta(t)$. Then, we have
$$\sup_{v \in \bfR^3} \left | t^2 E(t,x+vt) - E_\infty(v) \right |  \lesssim \| \mcP(t) - \mcP_\infty\|_\infty + t^{-1}  \alt{\mu(t)}^{2/3} \mcG(t) \left (\alt{\mu(t)}^{1/3} + \beta(t) \right ).$$
\end{lemma}
\begin{proof}
We begin by proving the latter result, and note that the former will follow with minor alterations.
In order to properly decompose the difference of these quantities, we first compute the translated field. In particular, we have
$$E(t,x+vt) = \sum_{\alpha = 1}^N \alt{\frac{q_\alpha}{4\pi}} \iint \frac{x +vt - y}{\left | x + vt - y \right |^3} \ g^\alpha(t, y - ut, u) \ du dy,$$
which, upon performing the change of variables
$$\xi = \frac{x+vt - y}{t}$$
with respect to $y$
and
$$w = x + (v-u-\xi)t$$
with respect to $u$, gives
$$E(t,x+vt) = \alt{\frac{1}{4\pi t^2}}  \sum_{\alpha = 1}^N \alt{q_\alpha} \iint \frac{\xi}{|\xi|^3} \ g^\alpha \left (t, w, v - \xi + \frac{x-w}{t} \right ) \ dwd\xi.$$
Therefore, due to the convolution structure of $E_\infty$ we have
\begin{equation}
\label{Dform}
t^2 E(t,x+vt) - E_\infty(v) =  \alt{\frac{1}{4\pi}}\int \frac{\xi}{|\xi|^3}  \sum_{\alpha = 1}^N \alt{q_\alpha} \left ( \int g^\alpha \left (t, w, v - \xi + \frac{x-w}{t} \right )  dw- F^\alpha_\infty(v - \xi) \right )\ d\xi.
\end{equation}
Next, we split the $\xi$-integrand so that
$$ \sum_{\alpha = 1}^N \alt{q_\alpha} \left ( \int g^\alpha \left (t, w, v - \xi + \frac{x-w}{t} \right ) \ dw - F^\alpha_\infty(v - \xi) \right ) =: \mcA_1(t,v-\xi) + \mcA_2(t,x,v-\xi)$$
where
\begin{equation}
\label{G1}
\mcA_1(t,v) =  \sum_{\alpha = 1}^N \alt{q_\alpha} \left (\int g^\alpha(t,w, v) dw - F^\alpha_\infty(v) \right ) = \mcP(t,v) - \mcP_\infty(v)
\end{equation}
and
\begin{equation}
\label{G2}
\mcA_2(t,x,v) = \sum_{\alpha = 1}^N \alt{q_\alpha}\int \left ( g^\alpha \left (t, w, v  + \frac{x-w}{t} \right ) -g^\alpha(t,w, v) \right )dw.
\end{equation}
Using this decomposition in \eqref{Dform}, we have
$$\sup_{v \in \bfR^3} \left |t^2 E(t,x+vt) - E_\infty(v) \right | \leq  \left \| \nabla_v (\Delta_v)^{-1} \mcA_1(t) \right \|_{L^\infty_v} + \left \| \nabla_v (\Delta_v)^{-1} \mcA_2(t,x) \right \|_{L^\infty_v}.$$
To estimate the convolution terms on the right side of the inequality, we will use Lemma \ref{LE}.

Now, to estimate the $\mcA_1$ term we find
$$\|\mcA_1(t) \|_{L^1_v} = \int \left | \mcP(t,\xi) - \mcP_\infty(\xi) \right | \ d\xi \lesssim \|\mcP(t) - \mcP_\infty \|_\infty $$
as the integral in $\xi$ is over a compact subset of $\bfR^3$ even as $t \to \infty$ due to Lemma \ref{XVsupp}.
Of course, we also have
$$\|\mcA_1(t) \|_{L^\infty_v} =  \| \mcP(t) - \mcP_\infty \|_\infty.$$
Using these estimates with Lemma \ref{LE} yields
\begin{equation}
\label{I}
\| \nabla_v (\Delta_v)^{-1} \mcA_1(t) \|_{L^\infty_v} \lesssim \| \mcP(t) - \mcP_\infty \|_\infty.
\end{equation}

To control the $\mcA_2$ term, we use the spatial support of $g^\alpha$ and velocity derivatives, which yields
\begin{eqnarray*}
\| \mcA_2(t,x) \|_{L^\infty_v} 
& = & \sup_{v \in \bfR^3} \left | \sum_{\alpha = 1}^N \alt{q_\alpha} \int \left [g^\alpha \left (t, w, v  + \frac{x-w}{t} \right ) -  g^\alpha(t,w, v) \right ] \ dw \right |\\
& \lesssim & \sup_{v \in \bfR^3} \sum_{\alpha = 1}^N \int \left |\int_0^1 \frac{d}{d\theta} \left (g^\alpha \left (t, w, v + \theta \frac{x-w}{t} \right )  \right )  d\theta \right |  dw\\
& \lesssim & t^{-1} \sup_{v \in \bfR^3} \sum_{\alpha = 1}^N \int_0^1 \int |x-w| \left |\nabla_v g^\alpha \left (t, w, v + \theta \frac{x-w}{t} \right ) \right |  dw d\theta\\
& \lesssim & t^{-1}  \sum_{\alpha = 1}^N\| \nabla_v g^\alpha (t) \|_\infty \int_{\alt{\{g^\alpha \neq 0\}}} \left ( |x| + |w| \right )  dw\\
& \lesssim & t^{-1} \alt{\mu(t)} \mcG(t) \left (\alt{\mu(t)}^{1/3}   + \beta(t) \right ) .  
\end{eqnarray*}
In order to estimate $\| \mcA_2(t,x) \|_{L^1_v}$, we proceed similarly but use Lemma \ref{MPP} so that $\alt{|\bS_g(t)| = |\bS_g(0)|}$ and 
\begin{eqnarray*}
\| \mcA_2(t,x) \|_{L^1_v}
& \lesssim & \sum_{\alpha = 1}^N \iint \left | g^\alpha \left (t, w, v + \frac{x-w}{t} \right ) - g^\alpha(t,w, v) \right | dw dv\\
& \lesssim & t^{-1} \sum_{\alpha = 1}^N \int_0^1 \iint |x-w| \left | \nabla_v g^\alpha \left (t, w, v + \theta \frac{x-w}{t} \right )\right | dv dw  d\theta\\
& \lesssim & t^{-1} \sum_{\alpha = 1}^N \left (\alt{\mu(t)}^{1/3} + \beta(t) \right )\ \iint\limits_{\alt{\bS_g(t)}} \left | \nabla_v g^\alpha \left (t, w, u \right )\right | du dw \\
& \lesssim & t^{-1} \mcG(t) \left (\alt{\mu(t)}^{1/3} + \beta(t) \right ).
\end{eqnarray*}
Combining this with the bound on $\| \mcA_2(t,x) \|_{L^\infty_v}$ within Lemma \ref{LE} gives
\begin{equation}
\label{II}
\| \nabla_v (\Delta_v)^{-1} \mcA_2(t,x) \|_{L^\infty_v} \lesssim t^{-1} \alt{\mu(t)}^{2/3} \mcG(t) \left (\alt{\mu(t)}^{1/3} + \beta(t) \right ).
\end{equation}
Finally, collecting \eqref{I} and \eqref{II}
we conclude
$$\sup_{v \in \bfR^3} \left |t^2 E(t,x+vt) - E_\infty(v) \right | \lesssim\| \mcP(t) - \mcP_\infty\|_\infty +  t^{-1}  \alt{\mu(t)}^{2/3} \mcG(t)\left (\alt{\mu(t)}^{1/3} + \beta(t) \right ).$$

Next, we prove the first conclusion.
In particular, the field is decomposed using the same change of variables so that
\begin{eqnarray*}
t^2 E(t,x) - E_\infty \left (\frac{x}{t} \right ) & = &  \alt{\frac{1}{4\pi}} \int \frac{\xi}{|\xi|^3}  \sum_{\alpha = 1}^N \alt{q_\alpha} \left ( \int g^\alpha \left (t, w, \frac{x-w}{t} - \xi \right )  dw- F^\alpha_\infty \left (\frac{x}{t} - \xi \right) \right )\ d\xi\\
& = &  \alt{\frac{1}{4\pi}}\int \frac{\xi}{|\xi|^3}  \left [ \mcA_1\left (t, \frac{x}{t}-\xi \right) + \tilde{\mcA}_2\left (t,\frac{x}{t}-\xi \right) \right ] \ d\xi
\end{eqnarray*}
where
$\mcA_1(t,v)$ is given by \eqref{G1}
and
\begin{equation}
\label{G2tilde}
\tilde{\mcA}_2(t,v) = \sum_{\alpha = 1}^N \alt{q_\alpha}\int \left ( g^\alpha \left (t, w, v - \frac{w}{t} \right ) -g^\alpha(t,w, v) \right )dw.
\end{equation}
We use the same strategy as before, invoking Lemma \ref{LE} to estimate the terms on the right side of %\eqref{YoungInt} 
$$\sup_{x \in \bfR^3} \left |t^2 E(t,x) - E_\infty \left (\frac{x}{t} \right ) \right | \leq  \left \| \nabla_v (\Delta_v)^{-1} \mcA_1(t) \right \|_{L^\infty_v} + \left \| \nabla_v (\Delta_v)^{-1} \tilde{\mcA}_2(t) \right \|_{L^\infty_v}.$$
The estimates of $\| \mcA_1(t) \|_{L^1_v}$ and $\| \mcA_1(t) \|_{L^\infty_v}$ are unchanged, while the estimates of $\| \tilde{\mcA}_2(t) \|_{L^1_v}$ and $\| \tilde{\mcA}_2(t) \|_{L^\infty_v}$ are analogous to those of $\mcA_2$.
Indeed, using the same approach we find
\begin{eqnarray*}
\| \tilde{\mcA}_2(t) \|_{L^\infty_v} 
& = & \sup_{v \in \bfR^3} \left | \sum_{\alpha = 1}^N \alt{q_\alpha} \int \left (g^\alpha \left (t, w, v - \frac{w}{t} \right ) -  g^\alpha(t,w, v) \right ) \ dw \right |\\
& \lesssim & t^{-1}  \sum_{\alpha = 1}^N\| \nabla_v g^\alpha (t) \|_\infty \int_{\alt{\{g^\alpha \neq 0\}}} |w|  dw\\
& \lesssim & t^{-1} \alt{\mu(t)}^{4/3} \mcG(t)
\end{eqnarray*}
and
$$ \| \tilde{\mcA}_2(t,x) \|_{L^1_v}
\lesssim \sum_{\alpha = 1}^N \iint \left | g^\alpha \left (t, w, v - \frac{w}{t} \right ) - g^\alpha(t,w, v) \right | dw dv\\
\lesssim t^{-1} \alt{\mu(t)}^{1/3} \mcG(t).$$
Ultimately, the result is 
$$\sup_{x \in \bfR^3} \left |t^2 E(t,x) - E_\infty \left (\frac{x}{t} \right ) \right | \lesssim \| \mcP(t) - \mcP_\infty\|_\infty + t^{-1} \alt{\mu(t)} \mcG(t).$$
\end{proof}

\begin{remark}
The dependence on $\mcG(t)$ can be removed from the field estimate by transferring a derivative from $g^\alpha$ onto the fundamental solution, similar to the forthcoming proof of Lemma \ref{LDE}, but one incurs an extra logarithmic factor and greater powers of $\alt{\mu(t)}$. In particular, the estimate becomes 
$$\sup_{x \in \bfR^3} \left | t^2 E(t,x) - E_\infty\left (\frac{x}{t} \right) \right | \lesssim \| \mcP(t) - \mcP_\infty\|_\infty + t^{-1} \ln(t) \alt{\mu(t)}^{4/3}.$$
\end{remark}

Next, we can estimate derivatives of the field and the charge density using the same tools.
\begin{lemma}
\label{LDE}
We have
$$\sup_{x \in \bfR^3} \left | t^3 \nabla_xE(t,x) - \nabla_vE_\infty\left (\frac{x}{t} \right) \right | \lesssim \ln(t) \left (\| \mcP(t) - \mcP_\infty\|_\infty + t^{-1} \alt{\mu(t)}^{4/3} \mcG(t) \right ),$$
Alternatively, assume that there is $\beta : [0,\infty) \to [0,\infty)$ such that $| x | \leq \beta(t)$. Then, we have
$$\sup_{v \in \bfR^3} \left | t^3 \nabla_x E(t,x+vt) - \nabla_v E_\infty(v) \right | \lesssim \ln(t)\left ( \| \mcP(t) - \mcP_\infty\|_\infty + t^{-1} \alt{\mu(t)} \mcG(t) \left (\alt{\mu(t)}^{1/3} + \beta(t) \right ) \right ).$$
\end{lemma}
\begin{proof}
As for the previous result, we first prove the latter conclusion, as it is more complicated.
We begin by taking the $v_k$-derivative of the $j$th component of \eqref{Dform} so that
\begin{align*}
& t^3 \partial_{x_k}E^j(t,x+vt)  - \partial_{v_k} E^j_\infty(v) \\
& =  \alt{\frac{1}{4\pi}} \int \frac{\xi_j}{|\xi|^3}  \sum_{\alpha = 1}^N \alt{q_\alpha} \left ( \int \partial_{v_k}g^\alpha \left (t, w, v - \xi + \frac{x-w}{t} \right )  dw- \partial_{v_k}F^\alpha_\infty(v - \xi) \right )\ d\xi\\
& = \alt{\frac{1}{4\pi}} \int \frac{\xi_j}{|\xi|^3} \left ( \partial_{v_k} \mcA_1(t,v-\xi) + \partial_{v_k}\mcA_2(t,x,v-\xi) \right ) \ d\xi
\end{align*}
where $\mcA_1(t,v)$ and $\mcA_2(t,x,v)$ are defined by \eqref{G1} and \eqref{G2}, respectively.
From these definitions we first note that
\begin{equation}
\label{H1}
\alt{\| \partial_{v_k} \mcA_1(t) + \partial_{v_k}\mcA_2(t,x) \|_{L^\infty_v} \lesssim \|\partial_{v_k} \mcP_\infty\|_\infty + \|\partial_{v_k} \mcP(t)\|_\infty + \alt{\mu(t)}\sum_{\alpha=1}^N \|\partial_{v_k} g^\alpha(t) \|_\infty \lesssim \alt{\mu(t)}\mcG(t),}
%\leq \sum_{\alpha=1}^N |\bS_x(t)| \| \nabla_v g^\alpha(t) \|_\infty + \| \partial_{v_k} \mcP_\infty \|_\infty 
\end{equation}
while using the estimates established by the proof of Lemma \ref{LField}, we have
\begin{equation}
\label{H2}
\| \mcA_1(t) \|_{L^\infty_v} + \|\mcA_2(t,x) \|_{L^\infty_v} \lesssim \| \mcP(t) - \mcP_\infty\|_\infty + t^{-1}  \alt{\mu(t)} \mcG(t) \left (\alt{\mu(t)}^{1/3} + \beta(t) \right )
\end{equation}
and
\begin{equation}
\label{H3}
\| \mcA_1(t) \|_{L^1_v} + \| \mcA_2(t,x) \|_{L^1_v} \lesssim \| \mcP(t) - \mcP_\infty\|_\infty +  t^{-1} \mcG(t) \left (\alt{\mu(t)}^{1/3} + \beta(t) \right ).  
\end{equation}
Then, decomposing the difference of derivatives and using $\partial_{v_k} \mcA_i = - \partial_{\xi_k} \mcA_i$ with an integration by parts away from the singularity, we find
\begin{eqnarray*}
\left |t^3 \partial_{x_k}E^j(t,x+vt)  - \partial_{v_k} E^j_\infty(v) \right | & \leq & \left | \int_{|\xi| < d} \frac{\xi_j}{|\xi|^3} \left ( \partial_{v_k} \mcA_1(t,v-\xi) + \partial_{v_k}\mcA_2(t,x,v-\xi) \right ) \ d\xi \right |\\
& &  + \left | \int_{|\xi| = d} \frac{\xi_j\xi_k}{|\xi|^4}  \left (\mcA_1(t,v-\xi) + \mcA_2(t,x,v-\xi) \right ) \ dS_\xi \right |\\
& &  + \left | \int_{d < |\xi| < R}  \partial_{\xi_k}  \left ( \frac{\xi_j}{|\xi|^3} \right )   \left (\mcA_1(t,v-\xi) + \mcA_2(t,x,v-\xi) \right ) \ d\xi \right |\\
& & + \left | \int_{|\xi| > R} \partial_{\xi_k}  \left ( \frac{\xi_j}{|\xi|^3} \right )  \left (\mcA_1(t,v-\xi) + \mcA_2(t,x,v-\xi) \right ) \ d\xi \right |\\
& =: &  I + II + III + IV.
\end{eqnarray*}
The estimate of $I$ merely uses \eqref{H1} so that
$$I \lesssim \alt{\mu(t)}\mcG(t) d.$$
We use \eqref{H2} to estimate $II$, which yields
$$II \lesssim \| \mcA_1(t)+ \mcA_2(t,x) \|_{L^\infty_v} \left ( \int_{|\xi| = d} |\xi|^{-2} dS_\xi \right ) \lesssim \| \mcP(t) - \mcP_\infty\|_\infty +  t^{-1}  \alt{\mu(t)} \mcG(t) \left (\alt{\mu(t)}^{1/3} + \beta(t) \right ).$$
Similarly, \eqref{H2} is used to estimate $III$ as
\begin{eqnarray*}
III & \lesssim & \int_{d < |\xi| < R} |\xi|^{-3}  \left |\mcA_1(t,v-\xi) + \mcA_2(t,x,v-\xi) \right | \ d\xi\\
& \lesssim & \| \mcA_1(t)+ \mcA_2(t,x) \|_{L^\infty_v} \ln \left (\frac{R}{d}\right ) \\
& \lesssim & \ln \left (\frac{R}{d}\right ) \left ( \| \mcP(t) - \mcP_\infty\|_\infty +  t^{-1}  \alt{\mu(t)} \mcG(t) \left (\alt{\mu(t)}^{1/3} + \beta(t) \right ) \right ).
\end{eqnarray*}
Finally, we use \eqref{H3} to estimate $IV$ so that
\begin{eqnarray*}
IV & \lesssim & \int_{|\xi| > R} |\xi|^{-3}  \left |\mcA_1(t,v-\xi) + \mcA_2(t,x,v-\xi) \right | \ d\xi\\
& \lesssim & R^{-3} \| \mcA_1(t)+ \mcA_2(t,x) \|_{L^1_v}\\
& \lesssim & R^{-3} \left ( \| \mcP(t) - \mcP_\infty\|_\infty + t^{-1} \mcG(t) \left (\alt{\mu(t)}^{1/3} + \beta(t) \right ) \right ).  
\end{eqnarray*}
Collecting these estimates and choosing $d = t^{-1}$ with $R^{-3} = \ln(t)$ implies $\ln \left (\frac{R}{d} \right ) \lesssim \ln(t)$ and yields
$$\left |t^3 \partial_{x_k}E^j(t,x+vt)  - \partial_{v_k} E^j_\infty(v) \right | \lesssim \ln(t) \left ( \| \mcP(t) - \mcP_\infty\|_\infty +  t^{-1} \alt{\mu(t)} \mcG(t) \left (\alt{\mu(t)}^{1/3} + \beta(t) \right ) \right ).$$  

Finally, the first conclusion follows by making straightforward changes analogous to the previous lemma.
More specifically,
the derivatives are decomposed into
\begin{align*}
& t^3 \partial_{x_k}E^j(t,x)  - \partial_{v_k} E^j_\infty\left (\frac{x}{t} \right) \\
& = \alt{\frac{1}{4\pi}} \int \frac{\xi_j}{|\xi|^3}  \sum_{\alpha = 1}^N \alt{q_\alpha} \left ( \int \partial_{v_k}g^\alpha \left (t, w, \frac{x-w}{t} - \xi \right )  dw- \partial_{v_k}F^\alpha_\infty \left (\frac{x}{t} - \xi \right) \right )\ d\xi\\
& = \alt{\frac{1}{4\pi}} \int \frac{\xi_j}{|\xi|^3} \left [ \partial_{v_k} \mcA_1\left (t, \frac{x}{t} -\xi \right) + \partial_{v_k}\tilde{\mcA}_2 \left (t,\frac{x}{t}-\xi \right) \right ] \ d\xi
\end{align*}
where $\mcA_1(t,v)$ and $\tilde{\mcA_2}(t,v)$ are defined by \eqref{G1} and \eqref{G2tilde}, respectively.
We then use the estimates of $\mcA_1$ and $\tilde{\mcA}_2$ presented in Lemma \ref{LField}.
Otherwise, the proof is identical to the latter conclusion.
\end{proof}

\begin{lemma}
\label{LDensity}
We have
$$\sup_{x \in \bfR^3} \left | t^3 \rho(t,x) - \mcP_\infty \left (\frac{x}{t} \right) \right | \lesssim  \| \mcP(t) - \mcP_\infty\|_\infty + t^{-1} \alt{\mu(t)}^{4/3} \mcG(t).$$
%and in particular
%$$\sup_{x \in \bfR^3} \left | t^3 \rho(t,x) - \mcP_\infty \left (\frac{x}{t} \right) \right | \lesssim t^{-1}\ln^6(t).$$
\end{lemma}
\begin{proof}
As for the field, we must rewrite this difference in terms of the translated distribution functions. To this end, we have
$$\rho(t,x) =  \sum_{\alpha = 1}^N \alt{q_\alpha} \int g^\alpha (t, x - ut, u) \ du,$$
and, upon performing the change of variables
$$y = x-ut$$
with respect to $u$, we find 
$$\rho(t,x) = \frac{1}{t^3}  \sum_{\alpha = 1}^N \alt{q_\alpha} \int g \left (t, y, \frac{x-y}{t} \right ) \ dy.$$
Hence, the difference of the densities can be split into two terms as
\begin{eqnarray*}
\left | t^3 \rho(t,x) - \mcP_\infty \left (\frac{x}{t} \right) \right | &\leq & \left |  \sum_{\alpha = 1}^N  \alt{q_\alpha}\int \left [g \left (t, y, \frac{x-y}{t} \right ) - g\left (t, y, \frac{x}{t} \right ) \right]\ dy \right | \\ 
& \ &  + \left | \mcP \left (t, \frac{x}{t} \right) - \mcP_\infty \left (\frac{x}{t} \right ) \right|\\
& =: & I + II.
\end{eqnarray*}

Using methods similar to the previous two lemmas, the first term satisfies
$$I \lesssim t^{-1}  \sum_{\alpha = 1}^N \| \nabla_vg^\alpha(t) \|_\infty  \int_{\alt{\{g \neq0 \}}} |y| \ dy \lesssim t^{-1} \alt{\mu(t)}^{4/3}\mcG(t),$$
while the second term is straightforward, namely
$$II \leq %\sup_{w \in \bfR^3} \left | \mcP(t,w) - \mcP_\infty(w) \right | = 
\| \mcP(t) - \mcP_\infty \|_\infty.$$
% \lesssim t^{-1}\ln^5(t).$$
Combining these estimates then yields the stated result.
\end{proof}

Finally, we estimate the current density in a similar fashion.
\begin{lemma}
\label{LCurrent}
We have
$$\sup_{x\in \bfR^3} \left | t^3 j(t,x) - \frac{x}{t} \mcP_\infty\left ( \frac{x}{t} \right) \right | \lesssim
\| \mcP(t) - \mcP_\infty \|_\infty + 
t^{-1}\alt{\mu(t)}^{4/3} \mcG(t).$$
%and in particular
%$$\sup_{x\in \bfR^3} \left | t^3 j(t,x) - \frac{x}{t} \mcP_\infty\left ( \frac{x}{t} \right) \right | \lesssim t^{-1}\ln^6(t).$$
\end{lemma}
\begin{proof}
%We first rewrite the current in terms of the translated distribution function so that
%$$j(t,x+vt) =  \sum_{\alpha = 1}^N \alt{q_\alpha} \int u g^\alpha(t, x+vt - ut, u) \ du.$$
%Upon performing the change of variables
%$$y = x+(v-u)t$$
%with respect to $u$, we find 
First, we note that due to Lemma \ref{XVsupp} and the compact support of $\mcP_\infty$, for $t$ sufficiently large there is $C > 0$ such that  $j(t,x) = \mcP_\infty \left (\frac{x}{t} \right ) = 0$ when $|x| \geq Ct$. Hence, it suffices to take $|x| \lesssim t$ throughout.
Next, we perform the same change of variables as in Lemma \ref{LDensity} to arrive at
$$j(t,x) = \frac{1}{t^3}  \sum_{\alpha = 1}^N \alt{q_\alpha} \int \left ( \frac{x-y}{t} \right ) g^\alpha \left (t, y, \frac{x-y}{t} \right ) \ dy.$$
Hence, the difference can be split into three terms as
\begin{eqnarray*}
\left | t^3 j(t,x) - \frac{x}{t}\mcP_\infty\left ( \frac{x}{t} \right) \right | &\leq & \left | \sum_{\alpha = 1}^N \alt{q_\alpha} \int \frac{y}{t}g^\alpha \left (t, y, \frac{x-y}{t} \right ) \ dy \right |\\
& \ & +  \left |  \sum_{\alpha = 1}^N \alt{q_\alpha} \int \frac{x}{t} \left [ g^\alpha \left (t, y, \frac{x-y}{t} \right )  - g^\alpha \left ( t, y, \frac{x}{t} \right) \right ] \ dy\right | \\
& \ & + \left | \frac{x}{t} \left (\mcP \left (t,\frac{x}{t} \right) - \mcP_\infty \left (\frac{x}{t} \right) \right ) \right|\\
& =: & I + II + III.
\end{eqnarray*}

The first term is estimated using the $L^\infty$ bound on $g^\alpha$ so that
$$I \lesssim t^{-1}  \sum_{\alpha = 1}^N \| g^\alpha(t) \|_\infty  \int_{\alt{\{ g^\alpha \neq 0\}}} |y| \ dy \lesssim t^{-1} \alt{\mu(t)}^{4/3}.$$
The second term has similar structure, but involves derivatives of $g^\alpha$, and we find
$$II\lesssim |x| t^{-2}  \sum_{\alpha = 1}^N \| \nabla_v g^\alpha(t) \|_\infty  \int_{\alt{\{ g^\alpha \neq 0\}}} |y| \ dy \lesssim t^{-1} \alt{\mu(t)}^{4/3} \mcG(t).$$
Finally, the third term is straightforward and yields
$$III \lesssim |x| t^{-1} \| \mcP(t) - \mcP_\infty \|_\infty \lesssim  \| \mcP(t) - \mcP_\infty \|_\infty .$$
Combining these estimates then yields the stated result.
\end{proof}

\section{Spatial Limits and Modified Scattering }
\label{sect:modscatt}

With the field and derivative estimates solidified, we prove that the distribution function scatters to a limiting value as $t \to \infty$ along a specific trajectory in phase space that may differ from its linear profile, and this is known as ``modified scattering''.
Many of the ideas in this direction arise from \cite{Ionescu}. 
We also mention \cite{CK}, which arrived at a modified scattering result but without an explicit representation for the associated trajectories.
First, \alt{we remind the reader} that using the field estimate of Lemma \ref{XVsupp} in Lemma \ref{Xsupp} produces
\begin{equation}
\label{Y}
|\mcY(t)| \lesssim \ln(t)
\end{equation} 
for any $\alpha = 1, ..., N$\alt{, where the implicit constant in this inequality may depend upon fixed $\tau \geq 0$ and $(x,v) \in \mcU$}. 
Thus, the characteristics of $g^\alpha$ may grow unbounded in time and require an additional logarithmic correction to construct spatial trajectories that converge as $t \to \infty$. To this point, we define for $\tau, t \geq 1$
\begin{align*}
\mcZ(t,\tau,x,v) & = \mcY(t, \tau,x,v) + \alt{\frac{q_\alpha}{m_\alpha}} \ln(t) E_\infty(\mcV(t, 
\tau, x,v))\\
& = \mcX(t,\tau, x, v) - t \mcV(t, \tau, x, v) + \alt{\frac{q_\alpha}{m_\alpha}} \ln(t) E_\infty(\mcV(t, 
\tau, x,v)).
\end{align*}
Then, we have the following result as $t \to \infty$.
\begin{lemma}
\label{LModChar}
For any $\alpha = 1, ..., N$, $\tau \geq 1$ and $(x, v) \in \mcU$, the limiting modified spatial characteristics $\mcZ_\infty$ defined by
$$\mcZ_\infty(\tau, x, v) :=  \lim_{t \to \infty} \mcZ(t, \tau, x, v)$$ 
exist, and are $C^2$, bounded, and invariant under the flow.
Additionally, we have the convergence estimates
$$|\mcZ(t, \tau, x, v) - \mcZ_\infty(\tau, x, v) | \lesssim t^{-1}\ln^5(t).$$
and
$$\left |\mcZ_\infty(\tau, x, v) - \left (x-v\tau + \alt{\frac{q_\alpha}{m_\alpha}} \ln(\tau) E_\infty(v) \right ) \right | \lesssim \tau^{-1}\ln^5(\tau).$$
Furthermore, there is $T_2 > T_1$ such that for all $\tau \geq T_2$ and $(x, v) \in \mcU$, we have
$$\left |\det \left (\frac{\partial (\mcZ_\infty, \mcV_\infty)}{\partial(x,v)}(\tau, x, v) \right ) \right | \geq \frac{1}{2}.$$
Consequently, for $\tau \geq T_2$ and $(x, v) \in \mcU$, the $C^2$ mapping $(x,v) \mapsto (\mcZ_\infty, \mcV_\infty)(\tau, x, v)$ is injective and invertible.
\end{lemma}
\begin{proof}
We begin by using \eqref{Sx}, \eqref{G}, \eqref{Y}, and Corollary \ref{P} within Lemmas \ref{LField} and \ref{LDE} (where $\beta(t) = C\ln(t)$) to deduce
\begin{equation}
\label{E_est}
\left |t^2 E \left (t,\mcY(t) + t\mcV(t) \right ) - E_\infty(\mcV(t)) \right |  \lesssim t^{-1} \ln^5(t)
\end{equation}
and
\begin{equation}
\label{DE_est}
\left | t^3 \nabla_xE \left (t,\mcY(t) + t\mcV(t) \right ) - \nabla_v E_\infty(\mcV(t)) \right | \lesssim t^{-1} \ln^7(t),
\end{equation}
respectively.

Now, to prove the convergence result, we merely need to demonstrate the integrability of $| \dot{\mcZ}(t) |$.
Indeed, we use Lemma \ref{XVsupp} and \eqref{E_est} to arrive at
\begin{eqnarray*}
\left | \dot{\mcZ}(t) \right | & = & \left |-\alt{\frac{q_\alpha}{m_\alpha}} t E \left (t,\mcY(t) + t\mcV(t) \right ) + \alt{\frac{q_\alpha}{m_\alpha}} t^{-1} E_\infty(\mcV(t)) \right .\\
& & \quad  + \ln(t) \nabla E_\infty(\mcV(t)) E(t, \mcY(t) + t\mcV(t))\biggr |\\
& \lesssim & t^{-1} \left |t^2 E \left (t,\mcY(t) + t\mcV(t) \right ) - E_\infty(\mcV(t)) \right | + \ln(t) \Vert \nabla E_\infty \Vert_\infty \Vert E(t)\Vert_\infty\\
& \lesssim & t^{-2}\ln^5(t).
\end{eqnarray*}
Therefore, using the initial conditions, $\mcZ(t,\tau,x,v)$ converges to a function 
$$\mcZ_\infty(\tau, x,v):= x-v\tau + \alt{\frac{q_\alpha}{m_\alpha}} \ln(\tau) E_\infty(v) + \int_\tau^\infty \dot{\mcZ}(s,\tau,x,v) \ ds $$ 
as $t \to \infty$, and the stated error estimates follow immediately.

Estimates on derivatives of $\mcZ_\infty$ are analogous to that of Lemma \ref{Vinfinvert}. 
Using the estimate of field derivatives from Lemma \ref{DEf} within Lemma \ref{dY} produces
\begin{equation}
\label{dY_est}
\left | \frac{\partial \mcY}{\partial x}(t,\tau, x, v) \right | \lesssim \ln^2(t) \qquad \mathrm{and} \qquad
\left | \frac{\partial \mcY}{\partial v}(t,\tau, x, v) \right | \lesssim \ln^2(t).
\end{equation}
Therefore, taking an $x$ derivative in the above expression for $\dot{\mcZ}(t)$, we use Lemmas \ref{XVsupp}, \ref{DEf}, and \ref{LDchar}, as well as \eqref{DE_est} and \eqref{dY_est} to find
\begin{eqnarray*}
\left |\frac{\partial \dot{\mcZ}}{\partial x} (t) \right | & \leq & \left |\alt{\frac{q_\alpha}{m_\alpha}} t \nabla_x E(t, \mcY(t) + t\mcV(t)) \left ( \frac{\partial \mcY}{\partial x}(t) + t  \frac{\partial \mcV}{\partial x}(t) \right ) - \alt{\frac{q_\alpha}{m_\alpha}} t^{-1} \nabla_vE_\infty(\mcV(t)) \frac{\partial \mcV}{\partial x}(t) \right |\\
& \ & + \ln(t) \left |\nabla_v^2 E_\infty(\mcV(t)) \right |  \left | E(t, \mcY(t) + t \mcV(t)) \right | \left | \frac{\partial \mcV}{\partial x}(t)  \right | \\
& \ & + \ln(t) \left |\nabla_v E_\infty(\mcV(t)) \right |  \left | \nabla_x E(t, \mcY(t) + t \mcV(t)) \right | \left | \frac{\partial \mcY}{\partial x}(t) + t  \frac{\partial \mcV}{\partial x}(t) \right | \\
& \leq &  Ct^{-1} \left | t^3 \nabla_xE \left (t,\mcY(t) + t\mcV(t) \right ) - \nabla_v E_\infty(\mcV(t)) \right | \left | \frac{\partial \mcV}{\partial x}(t) \right | + Ct^{-2}\ln^3(t)\\
& \lesssim & t^{-2}\ln^7(t).
\end{eqnarray*}
Integrating then yields
$$\left |\frac{\partial \mcZ_\infty}{\partial x} (\tau,x,v) - \mathbb{I} \right | \leq \int_\tau^\infty \left | \frac{\partial \dot{\mcZ}}{\partial x}(s) \right | \ ds  \lesssim \tau^{-1}\ln^7(\tau).$$
We repeat this argument for $v$ derivatives using the second estimate of \eqref{dY_est} to similarly arrive at
$$\left |\frac{\partial \dot{\mcZ}}{\partial v} (t) \right | \lesssim t^{-2}\ln^7(t),$$
and thus 
$$\left |\frac{\partial \mcZ_\infty}{\partial v} (\tau,x,v) + \tau\mathbb{I} - \alt{\frac{q_\alpha}{m_\alpha}} \ln(\tau) \nabla E_\infty(v) \right | \lesssim \tau^{-1}\ln^7(\tau).$$
This bound then implies $\left |\frac{\partial \mcZ_\infty}{\partial v} (\tau,x,v) \right| \lesssim \tau.$

Finally, using these estimates and those of Lemma \ref{Vinfinvert}, we can prove the lower bound on the determinant.
In particular, due to the block structure of the matrix we have
$$J(\tau) := \frac{\partial (\mcZ_\infty, \mcV_\infty)}{\partial(x,v)}(\tau, x, v) = 
\begin{bmatrix}
\frac{\partial \mcZ_\infty}{\partial x} (\tau,x,v) & \frac{\partial \mcZ_\infty}{\partial v} (\tau,x,v)\\ \ \\ \frac{\partial \mcV_\infty}{\partial x} (\tau,x,v)& \frac{\partial \mcV_\infty}{\partial v} (\tau,x,v)
\end{bmatrix}
=:
\begin{bmatrix}
A(\tau) & B(\tau)\\ C(\tau) & D(\tau)
\end{bmatrix}$$
for $\tau \geq 1$ and $(x,v) \in \mcU$.
Due to Lemma \ref{Vinfinvert}, taking $\tau > T_1$ ensures that $D(\tau)$ is invertible, and as $D(\tau) \to \mathbb{I}$ as $\tau \to \infty$, we have $\det(D) \to 1$ as $\tau \to \infty$ due to the continuity of the determinant operator.
Therefore, using the Schur complement we find
$$\det(J) = \det(D) \det \left ( A - BD^{-1}C \right )$$
for $\tau > T_1$.
Using the estimates on $A(\tau)$ and $B(\tau)$ given above and those of Lemma \ref{Vinfinvert}, we have
$$\left | A(\tau) - \mathbb{I} \right | \lesssim \tau^{-1} \ln^7(\tau) \qquad \mathrm{and} \qquad \left | BD^{-1} C \right | \leq |B||D|^{-1}|C| \lesssim \tau^{-1} \ln(\tau).$$
Thus, $A - BD^{-1} C \to \mathbb{I}$ as $\tau \to \infty$, which implies $\det(J) \to 1$ as $\tau \to \infty$.
Hence, there is $T_2 > T_1$ such that
$$\left |\det \left (\frac{\partial (\mcZ_\infty, \mcV_\infty)}{\partial(x,v)}(\tau, x, v) \right ) \right | \geq \frac{1}{2}$$
for all $\tau \geq T_2$, and the regularity and invertibility of the map $(x,v) \mapsto(\mcZ_\infty, \mcV_\infty)$ follows.
\end{proof}

Analogous to the velocity limits, we next define the collection of all limiting positions on $\alt{\bS_f(t)}$.
Due to the invariance under the flow defined by \eqref{char}, we have
$$\left \{ \mcZ_\infty(\tau, x, v) : (x, v) \in \alt{\bS_f(\tau)} \right \} = \left \{ \mcZ_\infty(1, x, v) : (x, v) \in \alt{\bS_f(1)} \right \}$$
for all $\tau \geq 0$. Hence, for any $\alpha = 1, ..., N$ define
$$\Omz := \left \{ \mcZ_\infty(1, x, v) : (x, v) \in \alt{\bS_f(1)} \right \}.$$
Additionally, as $\mcZ_\infty (1,x, v)$ is continuous due to Lemma \ref{LModChar},  its range $\Om_z^\alpha$ on the compact set $\alt{\bS_f(1)}$ is compact.
Now that we have shown the convergence of modified spatial characteristics, we can prove the convergence of the particle distribution functions along these trajectories.

\begin{lemma}
\label{LModScattering}
For every $\alpha = 1, ..., N$ define $\Om^\alpha = \Omz \times \Omv$.
There exists $f^\alpha_\infty \in C_c^2(\bfR^6)$ with $\mathrm{supp}(f^\alpha_\infty) = \Om^\alpha$ such that
$$h^\alpha(t,x,v) = f^\alpha \left (t,x + vt - \alt{\frac{q_\alpha}{m_\alpha}} \ln(t) E_\infty(v), v \right)$$
satisfies 
$$h^\alpha(t,x, v) \to f^\alpha_\infty(x,v)$$ 
uniformly as $t \to \infty$.
\alt{Moreover}, we have the convergence estimate
$$\sup_{(x,v) \in \bfR^6} \left | f^\alpha \left (t,x + vt - \alt{\frac{q_\alpha}{m_\alpha}} \ln(t)E_\infty(v), v \right ) - f^\alpha_\infty(x,v) \right |  \lesssim t^{-1}\ln^5(t),$$
and the limiting distribution preserves the \alt{particle number, velocity}, and energy of the system, i.e.
$$\iint f^\alpha_\infty(x,v) \ dvdx = \mcM^\alpha, \quad \iint v f^\alpha_\infty(x,v) \ dvdx = \mcJ^\alpha, \quad \mathrm{and} \quad 
\sum_{\alpha = 1}^N \iint \frac{1}{2} \alt{m_\alpha} |v|^2 f^\alpha_\infty(x,v) \ dv dx= \mcEVP.$$
\end{lemma}

\begin{proof}
\alt{We first establish the weak convergence of the distribution function and properties of the associated limit.}
Let $\psi  \in C_b(\bfR^6)$ be given and fix any $T > T_2 > T_1$ from Lemmas \ref{Vinfinvert} and \ref{LModChar}.
Then, we apply the well-known measure-preserving change of variables $(\tilde{x}, \tilde{v}) = (\mcX(T,t,x,v), \mcV(T, t, x, v))$, so that
\begin{align*}
& \lim_{t \to \infty} \iint \psi(x,v) h^\alpha(t,x,v) \ dv dx\\ 
& \qquad = \lim_{t \to \infty} \iint \psi(x,v) f^\alpha \left (t, x +vt - \alt{\frac{q_\alpha}{m_\alpha}} \ln(t) E_\infty(v), v \right ) \ dv dx\\
& \qquad  = \lim_{t \to \infty} \iint\limits_{\alt{\bS_f(t)}} \psi \left (x -vt  + \alt{\frac{q_\alpha}{m_\alpha}} \ln(t) E_\infty(v),v \right) f^\alpha(t, x, v) \ dv dx\\
& \qquad  = \lim_{t \to \infty} \iint\limits_{\alt{\bS_f(t)}} \psi \left(x -vt +\alt{\frac{q_\alpha}{m_\alpha}} \ln(t) E_\infty(v),v \right) f^\alpha(T, \mcX(T,t,x, v), \mcV(T,t,x, v)) \ dv dx\\
& \qquad  = \lim_{t \to \infty}\iint\limits_{\alt{\bS_f(T)}} \psi(\mcZ(t, T, \tilde{x}, \tilde{v}), \mcV(t,T,\tilde{x}, \tilde{v})) f^\alpha(T, \tilde{x}, \tilde{v}) \ d\tilde{v} d\tilde{x}\\
& \qquad  = \iint\limits_{\alt{\bS_f(T)}} \psi(\mcZ_\infty(T,\tilde{x}, \tilde{v}), \mcV_\infty(T,\tilde{x}, \tilde{v})) f^\alpha(T, \tilde{x}, \tilde{v}) \ d\tilde{v} d\tilde{x}
\end{align*}
by Lebesgue's Dominated Convergence Theorem. Now, by Lemma \ref{LModChar} for any $(\tilde{x}, \tilde{v}) \in \alt{\bS_f(T)}$, the mapping 
$$(\tilde{x}, \tilde{v}) \mapsto \left (\mcZ_\infty(T, \tilde{x}, \tilde{v}), \mcV_\infty(T, \tilde{x}, \tilde{v}) \right )$$ is $C^2$ with 
$$\left |\det \left (\frac{\partial (\mcZ_\infty, \mcV_\infty)}{\partial (x,v)}\right )(T, \tilde{x}, \tilde{v})\right |
\geq \frac{1}{2},$$ 
and thus bijective from $\alt{\bS_f(T)}$ to $\Om^\alpha$. 
Hence, letting $(y,u) = \left (\mcZ_\infty(T, \tilde{x}, \tilde{v}), \mcV_\infty(T, \tilde{x}, \tilde{v}) \right )$, we perform a change of variables and drop the tilde notation to find
\begin{align*}
& \lim_{t \to \infty} \int \psi(x,v) h^\alpha(t,x,v) \ dv dx\\
& = \iint  \psi(y,u) f^\alpha(T, Z^\alpha(y,u), V^\alpha(y,u)) \chfn_{\Om^\alpha}(y,u) \mcD^\alpha(y,u) \ du dy
\end{align*}
where 
$$\mcD^\alpha(y,u) = \left | \det \left (\frac{\partial (\mcZ_\infty, \mcV_\infty)}{\partial(x,v)} \right )(T, Z^\alpha(y,u), V^\alpha(y,u))\right |^{-1}$$
and the $C^2$ function $(Z^\alpha, V^\alpha): \Om^\alpha \to \alt{\bS_f(T)}$ is given by the components
$$Z^\alpha(y,u) = \left ( \mcZ_\infty \right ) ^{-1}(T,y,u)$$
and
$$V^\alpha(y,u) = \left (\mcV_\infty\right )^{-1}(T,y,u),$$
respectively.
Therefore, for any $(y,u) \in \bfR^6$ define 
$$f^\alpha_\infty(y,u) = f^\alpha(T, Z^\alpha(y,u), V^\alpha(y,u)) \chfn_{\Om^\alpha}(y,u) \mcD^\alpha(y, u).$$
%and notice
%$$| f^\alpha_\infty(y,u) | \leq 2\|f_0\|_\infty.$$
Thus, $f^\alpha_\infty \in C_c^2(\bfR^6)$ and satisfies
\begin{equation}
\label{Weaklimit2}
\lim_{t \to \infty} \iint \psi(x,v) h^\alpha(t,x,v)   dv dx = \iint \psi(y,u) f^\alpha_\infty(y,u)  du dy.
\end{equation}
Notice further that due to the compact support and regularity of $f^\alpha(t)$ and $f^\alpha_\infty$, it is sufficient to take $\psi \in C(\mcU)$, where $\mcU \subset \bfR^6$ is compact, for this equality to hold.
Finally, the conservation laws are maintained in the limit by merely choosing $\psi(x,v) = 1$, $\psi(x,v) = v$, and $\psi(x,v) = \frac{1}{2} \alt{m_\alpha}|v|^2$ within \eqref{Weaklimit}. The arguments are analogous to those of Lemma \ref{Funif}.

\alt{Next, we establish the uniform convergence of this function using the compact support of the limit.}
Fix $\alpha = 1, ..., N$ and for brevity let
$$\mcW(t,x,v) = x - \alt{\frac{q_\alpha}{m_\alpha}} \ln(t)E_\infty(v)$$
for every $t \geq 1$ and $x,v \in \bfR^3$ so that
$$h^\alpha(t,x,v) := f^\alpha \left (t,x + vt - \alt{\frac{q_\alpha}{m_\alpha}} \ln(t)E_\infty(v), v \right ) = g^\alpha(t, \mcW(t,x,v), v).$$
Throughout, we will assume $t \geq 1$ is suitably large.
Due to \eqref{Y} and the compact spatial support of $f^\alpha_\infty$ \alt{established above}, we note that there is $C>0$ 
such that $g^\alpha(t,\mcW(t,x,v),v) = f^\alpha_\infty(x,v)= 0$ whenever $|x| \geq C\ln(t)$. Hence, it suffices to take $|x| \lesssim \ln(t)$, and thus $|\mcW(t,x,v)| \lesssim \ln(t)$.
%Additionally, we will often use the properties of $g^\alpha$ established in Lemma \ref{Funif}.
%Indeed, 
Because $g^\alpha$ satisfies \eqref{VPg}, we deduce that $h^\alpha$ satisfies
$$\partial_t h^\alpha = \alt{\frac{q_\alpha}{m_\alpha}} t^{-1} \left ( t^2 E(t, \mcW+vt) - E_\infty(v) \right ) \cdot \nabla_x g^\alpha(t,\mcW,v) - \alt{\frac{q_\alpha}{m_\alpha}}E(t, \mcW+vt)\cdot \nabla_v g^\alpha(t,\mcW,v).$$
Similar to the proof of Lemma \ref{Funif}, we wish to show that 
$\displaystyle \| \partial_t h^\alpha(t) \|_\infty$ is integrable in order to establish the existence of a limiting function in this norm.

To this end, we decompose $\partial_t h^\alpha$ so that
$$\left |\partial_t h^\alpha(t,x,v) \right | \alt{\lesssim} I+ II$$
where
$$I = t^{-1}\left | \left ( t^2 E(t, \mcW+vt) - E_\infty(v) \right ) \cdot \nabla_x g^\alpha(t,\mcW,v)  \right |$$
and
$$II = \left | E(t, \mcW+vt)\cdot \nabla_v g^\alpha(t,\mcW,v) \right |.$$
The second term is well-behaved. Indeed, using Lemma \ref{XVsupp} and \eqref{G} we find
\begin{equation}
\label{IIh}
II \leq \|E(t) \|_\infty \mcG(t)  \lesssim t^{-2}\ln^2(t).
\end{equation}
Using Lemma \ref{DEf} we find for every $\alpha = 1, ..., N$
$$ \| \nabla_x g^\alpha(t) \|_\infty  = \| \nabla_x f^\alpha(t) \|_\infty \lesssim 1.$$
Thus, the latter term in $I$ is uniformly bounded in time, which implies
$$I \lesssim t^{-1} \left | t^2 E(t, \mcW+vt) - E_\infty(v) \right |.$$
%Because $(x,v) \in \bS(t)$ and $E_\infty(v)$ is bounded, we find
%$$\left |\mcW(t,x,v) \right | = \left | x - \alt{q_\alpha} \ln(t) E_\infty(v) \right | \lesssim \ln(t).$$  
Because $\left |\mcW(t,x,v) \right |  \lesssim \ln(t)$, we invoke Lemma \ref{LField} with $\beta(t) = C\ln(t)$ to conclude
\begin{equation}
\label{Ih}
I \lesssim t^{-2}\ln^5(t).
\end{equation}
Combining \eqref{IIh} and \eqref{Ih}, we have
$$ \|\partial_t h^\alpha(t) \|_\infty \lesssim t^{-2}\ln^5(t).$$
As this bound is integrable in time, there is $\mathfrak{f}^\alpha \in C(\bfR^6)$ such that
$$\| h^\alpha(t) - \mathfrak{f}^\alpha \|_\infty \lesssim t^{-1}\ln^5(t).$$
\alt{Of course}, the strong limit implies the weak limit, and we find $\mathfrak{f}^\alpha = f^\alpha_\infty$ by equivalence of weak limits.
\end{proof}

\section{Proofs of Theorems and Sharpness of Estimates}
\label{sect:prop}

We begin this section by proving the first theorem.
\begin{proof}[Proof of Theorem \ref{T1}]
The theorem follows by merely collecting the estimates of the previous sections.
Indeed, beginning with the decay estimate of the field \eqref{Assumption}, we conclude the improved decay estimates of  Lemma \ref{XVsupp}. Then, we invoke Lemmas \ref{DEf}, \ref{Xsupp}, and \ref{Dvg} to produce \eqref{Sx} and \eqref{G}.
With this, the spatial average of each species converges with the estimate provided in Section \ref{sect:vellim}.
Applying \eqref{Sx} and \eqref{G} to Lemmas \ref{LField}, \ref{LDE}, \ref{LDensity}, and \ref{LCurrent} provides the asymptotic behavior of the field, field derivatives, and spatial densities.
The estimates concerning the limits of modified spatial characteristics and the modified scattering result are provided in Section \ref{sect:modscatt}.
\alt{Additionally, Lemma \ref{LDE} removes the logarithmic factor from the original field derivative estimate of Lemma \ref{DEf}
so that
$$\| \nabla_x E(t) \|_\infty \leq t^{-3} \|E_\infty\|_\infty + t^{-4}\ln^7(t)\lesssim t^{-3}.$$
Using this in Lemma \ref{Dvg} then reduces the growth of velocity derivatives of $g^\alpha$ by one logarithmic power to 
$$\mcG(t) \lesssim \ln(t)$$
instead of \eqref{G}.
This estimate is then applied within Lemma \ref{Funif} and Corollary \ref{P}, as well as, Lemmas \ref{LField}, \ref{LDE}, \ref{LDensity}, \ref{LCurrent}, \ref{LModChar}, and \ref{LModScattering},
which provides the stated logarithmic powers and completes the proof.}
\end{proof}

Next, notice that if $\mcM \neq 0$, then the estimates presented in Theorem \ref{T1} completely characterize the asymptotic behavior of the field and associated densities. Indeed, \eqref{Pinfmass} yields
$$\int \mcP_\infty(v) \ dv = \mcM \neq 0,$$
which implies $\mcP_\infty \not\equiv 0$ and $E_\infty \not\equiv 0$.
Hence, the decay estimates are sharp, up to a a change in the logarithmic power.
This can be further demonstrated by the following result.

\begin{lemma}
If $\mcM \neq 0$, then 
$$  t^{-3} \lesssim \| \rho(t) \|_\infty  \qquad \mathrm{and} \qquad t^{-2} \lesssim \| E(t) \|_\infty.$$
\end{lemma}
\begin{proof}
The results follow merely from the conservation of \alt{charge} and the bound on spatial characteristics provided by Lemma \ref{XVsupp}. Indeed, we have
$$| \mcM | \leq \int_{|x| \leq Ct} | \rho(t,x) | \ dx \leq Ct^3 \| \rho(t) \|_\infty.$$
Due to the Divergence Theorem, the electric flux satisfies a similar estimate so that
$$| \mcM | =  \left | \int_{|x| \leq Ct} \nabla_x \cdot E(t,x) \ dx \right | \leq \int_{|x| = Ct} | E(t,x) | \ dS_x \leq Ct^2 \| E(t) \|_\infty.$$
Rearranging the inequalities and using $\mcM \neq 0$ then yields the stated estimates.
\end{proof}

However, when the plasma is neutral, i.e. $\mcM = 0$, it is possible that the limiting density and field are identically zero, which implies stronger decay properties for these quantities, as displayed in Theorem \ref{T2}.

\begin{proof}[Proof of Theorem \ref{T2}]
Beginning with the decay estimate of the field \eqref{Assumption}, we can repeat the argument as in the proof of Theorem \ref{T1} to arrive at the estimates stated therein. However, because $\mcP_\infty \equiv 0$, and thus $E_\infty \equiv 0$, the lemmas of Section \ref{sect:fieldconv} yield
$$\| E(t) \|_\infty \lesssim t^{-3} \ln^5(t)$$
from Lemma \ref{LField},
$$\| \nabla_x E(t) \|_\infty \lesssim t^{-4} \ln^7(t)$$
from Lemma \ref{LDE}, and
$$\| \rho(t) \|_\infty \lesssim t^{-4} \ln^6(t)$$
from Lemma \ref{LDensity}.
With these faster decay rates, Lemmas \ref{Xsupp} and \ref{Dvg} provide uniform-in-time bounds for the spatial support and velocity derivatives of $g^\alpha$ so that
\begin{equation}
\label{SG} 
\alt{\mu(t)} \lesssim 1 \qquad \mathrm{and} \qquad \mcG(t) \lesssim 1.
\end{equation}
Lemma \ref{Funif} then yields  
$$\| F^\alpha(t) - F^\alpha_\infty \|_\infty \lesssim t^{-2} \ln^6(t)$$
so that
$$\| \mcP(t) \|_\infty \lesssim t^{-2} \ln^6(t).$$
This estimate, used in conjunction with \eqref{SG}, then provides the stated decay rates of $\| E(t) \|_\infty$, $\| \nabla_x E(t) \|_\infty$, $\| \rho(t) \|_\infty$, and $\| j(t) \|_\infty$ by invoking Lemmas \ref{LField}, \ref{LDE}, \ref{LDensity}, and \ref{LCurrent}. In turn,
using the resulting rates in Lemma \ref{Funif} yields the stated estimates on $\| F^\alpha(t) - F^\alpha_\infty \|_\infty$ and $\| \mcP(t) \|_\infty$.
The improved rate of the field further produces the faster convergence rate of the velocity characteristics from Lemma \ref{L6}.

%Finally, the convergence of the translated spatial characteristics follows by applying the arguments of Section \ref{sect:modscatt} to $\mcY$ and noting that when $E_\infty \equiv 0$, we have $\mcY(t) = \mcZ(t)$.
Finally, the convergence of the translated spatial characteristics follows by applying the arguments of Section \ref{sect:modscatt} to $\mcY(t)$ for every $\alpha = 1, ..., N$ and using the improved decay of the field and its derivatives.
Indeed, from Lemma \ref{Xsupp}, we find
$$| \dot{\mathcal{Y}}^\alpha(t) | \lesssim t^{-2}$$
so that $\mcY(t)$ tends to a limit defined by
$$\mcY_\infty(\tau, x,v):= x-v\tau + \alt{\frac{q_\alpha}{m_\alpha}} \int_\tau^\infty s E(s, \mcX(s, \tau, x, v)) \ ds$$
with the convergence estimate
$$\left | \mcY(t,\tau, x, v) - \mcY_\infty(\tau, x, v) \right | \lesssim t^{-1}$$
for all $\tau \geq 0, (x,v) \in \mcU$.
We can then define
$$\Omy := \left \{ \mcY_\infty(0, x, v) : (x, v) \in \alt{\bS_f(0)} \right \}.$$
Using the methods of Lemma \ref{dY}, estimates on derivatives follow immediately, namely
$$\left | \frac{\partial \mcY_\infty}{\partial x}(\tau, x, v) - \mathbb{I} \right | \lesssim \int_\tau^\infty s^2 \| \nabla_x E(s) \|_\infty ds \lesssim \tau^{-1}\ln(\tau)$$
and
$$\left | \frac{\partial \mcY_\infty}{\partial v}(\tau, x, v)   + \tau \mathbb{I} \right | \lesssim \int_\tau^\infty s^2 \| \nabla_x E(s) \|_\infty \ ds \lesssim \tau^{-1}\ln(\tau)$$
so that $\left | \frac{\partial \mcY_\infty}{\partial v}(\tau, x, v) \right | \lesssim \tau$.
With these estimates, the determinant of the mapping $(x,v) \mapsto \left (\mcY_\infty(\tau,x,v), \mcV_\infty(\tau,x,v) \right )$ is bounded away from zero for $\tau$ sufficiently large using the Schur complement as in Lemma \ref{LModChar}.
Finally, the arguments proving the convergence of the translated distribution function $g^\alpha$ follow analogously to that of Lemma \ref{LModScattering}. In particular, from the Vlasov equation in \eqref{VPg} we find
$$|\partial_t g^\alpha(t,x,v) | \alt{\lesssim} \left |t E(t, x+vt) \cdot \nabla_x g^\alpha - E(t, x+vt) \cdot \nabla_v g^\alpha \right | \lesssim t\| E(t) \|_\infty + \| E(t) \|_\infty \mcG(t)$$
and thus
$$ \| \partial_t g^\alpha(t) \|_\infty \lesssim t^{-2}$$
for every $\alpha = 1, ..., N$,
which completes the proof.
\end{proof}

\section{\alt{Ackowledgements}}
\alt{The author would like to thank the anonymous reviewers for their valuable comments and suggestions to improve the paper.}

% Non-BibTeX users please use

\end{document}